\documentclass[journal]{IEEEtran}

\usepackage{amsmath,amssymb,bm,bbm,mathrsfs,amscd}
\usepackage{color}
\usepackage{amsthm}
\usepackage{dsfont}
\usepackage{graphicx}
\usepackage{epsfig}
\usepackage{subfigure}
\usepackage{algorithm}
\usepackage[noend]{algpseudocode}
\usepackage{tikz}
\usepackage{enumitem}
\usepackage{stackrel}
\usepackage{pifont}
\usepackage{hhline}
\usepackage{booktabs}
\usepackage{balance}

\newtheorem{theorem}{Theorem}
\newtheorem{definition}{Definition}
\newtheorem{proposition}{Proposition}
\newtheorem{lemma}{Lemma}
\newtheorem{corollary}{Corollary}
\newtheorem{remark}{Remark}

\newtheorem{assumption}{Assumption}

\newcommand{\RR}{{\mathbb{R}}}
\newcommand{\NN}{{\mathbb{N}}}
\newcommand{\EE}{{\mathbb{E}}}
\newcommand{\PP}{{\mathbb{P}}}
\newcommand{\JJ}{{\mathbb{J}}}
\newcommand{\FF}{{\mathbb{F}}}

\newcommand{\mc}{\mathcal}
\newcommand{\norm}[1]{\|#1\|}
\newcommand{\normsq}[1]{\|#1t\|^2}

\newcommand{\EEk}[1]{\EE\left[#1|\mc F_k\right]}
\newcommand{\EEx}[1]{\EE\left[#1\right]}
\newcommand{\op}{\operatorname}

\newcommand{\prox}{\operatorname{prox}}

\newcommand{\bs}{\boldsymbol}
\newcommand{\fineass}{\hfill\small$\square$}
\newcommand{\cmark}{\ding{51}}%
\newcommand{\xmark}{\ding{55}}%

\DeclareSymbolFont{myletters}{OML}{ztmcm}{m}{it}
\DeclareMathSymbol{\upomega}{\mathord}{myletters}{"21}

\ifCLASSINFOpdf
\else
\fi
\begin{document}

\title{Stochastic generalized Nash equilibrium seeking in merely monotone games}

\author{Barbara Franci and Sergio Grammatico
\thanks{The authors are in the Delft Center for System and Control, TU Delft, The Netherlands. E-mail addresses:
        {\tt\small \{b.franci-1, s.grammatico\}@tudelft.nl}.}
\thanks{This work was partially supported by NWO under research projects OMEGA (613.001.702) and P2P-TALES (647.003.003), and by the ERC under research project COSMOS (802348).}}

\markboth{Journal of \LaTeX\ Class Files,~Vol.~14, No.~8, August~2015}%
{Shell \MakeLowercase{\textit{et al.}}: Bare Demo of IEEEtran.cls for IEEE Journals}

\maketitle

\begin{abstract}
We solve the stochastic generalized Nash equilibrium (SGNE) problem in merely monotone games with expected value cost functions. 
Specifically, we present the first distributed SGNE seeking algorithm for monotone games that requires one proximal computation (e.g., one projection step) and one pseudogradient evaluation per iteration.

Our main contribution is to extend the relaxed forward--backward operator splitting by Malitsky (Mathematical Programming, 2019) to the stochastic case and in turn to show almost sure convergence to a SGNE when the expected value of the pseudogradient is approximated by the average over a number of random samples.



\end{abstract}


\IEEEpeerreviewmaketitle

\section{Introduction}

\IEEEPARstart{I}{n} a generalized Nash equilibrium problem (GNEP), some agents interact with the aim of minimizing their individual cost functions under some joint feasibility constraints. Due to the presence of the shared constraints, computing a GNE is usually hard. Despite this challenge, GNEPs has been studied extensively within the system and control community, for their wide applicability, e.g., in energy markets \cite{yi2019,kulkarni2012,pavel2019,kanzow2019,belgioioso2017}.

Unfortunately, the stochastic counterpart of GNEP is not studied as much \cite{ravat2011,koshal2013,yu2017,lei2020sync}. 
A stochastic GNEP (SGNEP) is a constrained equilibrium problem where the cost functions are expected value functions. Such problems arise when there is some uncertainty, expressed through a random variable with an unknown distribution. 
For instance, networked Cournot games with market capacity constraints and uncertainty in the demand can be modelled as SGNEPs \cite{demiguel2009,abada2013}.
Other examples arise in transportation systems \cite{watling2006}, and in electricity markets \cite{henrion2007}. 

If the random variable is known, the expected value formulation can be solved with a standard technique for the deterministic counterpart.
In fact, one possible approach for SGNEPs is to recast the problem as a stochastic variational inequality (SVI) through the use of the Karush-Kuhn-Tucker conditions. Then, the problem can be written as a monotone inclusion and solved via operator splitting techniques. To find a zero of the resulting operator, we propose a stochastic relaxed forward-backward (SRFB) algorithm. Our iterations are the stochastic counterpart of the golden ratio algorithm \cite{malitsky2019} for deterministic variational inequalities, which reduces to a stochastic relaxation of a forward-backward algorithm when applied to non-generalized Nash equilibrium problems.

Besides the shared constraints, the additional difficulty in \textit{stochastic} GNEPs is that the pseudogradient mapping is usually not directly accessible, for instance because the expected value is hard to compute. For this reason, often the search for a solution of a SVI relies on samples of the random variable.
Depending on the number of samples, there are two main methodologies available: stochastic approximation (SA) and sample average approximation (SAA). In the SA scheme \cite{robbins1951}, each agent samples one or a finite number of realizations of the random variable. While it can be computationally light, it may also require stronger assumptions on the mappings and on the parameters involved \cite{koshal2013,yousefian2017,yousefian2014}. To weaken the assumptions, it is often used in combination with the so-called variance reduction \cite{staudigl2019, iusem2017}, taking the average over an increasing number of samples. This approach, although it may be computationally costly, is used, for instance, in machine learning problems, where there is a huge number of data available. If the average is taken over an infinite number of samples instead, we have the SAA scheme \cite{shapiro2003}.

Independently of the approximation scheme, it is desirable to obtain distributed iterations, where each agent knows only its cost function and constraints \cite{yi2019,franci2019ecc}. In a full-decision information setting, the agents have access to the decisions of the other agents that affect their cost functions \cite{yi2019}, while in a partial-decision information setting the shared information is even more limited \cite{pavel2019}. In both cases, the key aspect is that the agents can communicate without the need for a central coordinator. Alternatively, a payoff-based information setup has been considered in \cite{tatarenko2019, bravo2018}, where the agents have access to the values of their own cost functions.

Besides being distributed, an algorithm for SGNEPs should converge under mild monotonicity assumptions and it should be relatively fast. For SVIs, there exist several methods that may be used for SGNEPs. Among others, one can consider the stochastic preconditioned forward--backward algorithm (SpFB) \cite{franci2019ecc,franci2019} for its convergence speed and low computational cost. The downside of this algorithm is that the pseudogradient mapping must be (monotone and) cocoercive \cite{rosasco2016,franci2019}, strongly monotone \cite{franci2019ecc,lei2020} or satisfy the variational stability \cite{mertikopoulos2019}. Similarly, one can consider the stochastic projected reflected gradient scheme (SPRG) \cite{cui2016,cui2019} that is fast but requires the weak sharpness property (implied by cocoercivity) which is, however, hard to check on the problem data. Nonetheless, weakening the assumption on the pseudogradient to mere monotonicity translates into having computationally expensive algorithms. In this case, one could apply the extragradient (EG) scheme \cite{iusem2017,kannan2019} with two projection steps per iteration or the forward--backward--forward (FBF) algorithm \cite{staudigl2019} that has one projection but two evaluations of the pseudogradient for each iteration. Recently, the stochastic subgradient EG (SSE) algorithm have been considered, that use one proximal step, but still two computations of the approximated pseudogradient \cite{cui2019}.
These considerations are summarized in Table \ref{table_algo}, where we consider variance-reduced schemes with \textit{fixed step sizes} for SGNEPs in comparison with our proposed SRFB algorithm. 
Essentially, for merely monotone games, our SRFB algorithm is the only one to perform one proximal step and one stochastic approximation of the pseudogradient mapping with fixed step size.
\begin{table}[h]
\begin{center}
\begin{tabular}{lcccccc}
\toprule
        & SFBF & SEG  & SSE  & SPRG & SpFB & SRFB  \\ 
        &  \cite{staudigl2019} &  \cite{iusem2017} &  \cite{cui2019}  &  \cite{cui2019} &  \cite{franci2019}& \\
\midrule
\textsc{Mon.}   &   \cmark  &  \cmark & \cmark & \xmark   &    \xmark  &   \cmark \\\midrule

\# $\mathrm{prox}$   & 1   & 2  & 1& 1 & 1    & 1\\ \midrule
\# $F$    & 2   & 2  & 2 &1  & 1    & 1\\ 
\bottomrule
\end{tabular}
\end{center}
\caption{The algorithms for SGNEPs that converge under only monotonicity (Mon.) are marked with \cmark. \# prox and \# F indicate the number of proximal steps and the number of computations of the stochastically approximated pseudogradient per iteration, respectively. All these algorithms use the variance reduction and a fixed step size.
}\label{table_algo}
\end{table}\vspace{-.6cm}

Another option that uses one projection and one computation of the pseudogradient is the iterative Tikhonov regularization \cite{koshal2013} which however is not proven to converge with variance reduction in SGNEPs and uses vanishing step sizes and vanishing regularization coefficients. Other algorithms have been proposed for saddle points problems (without coupling constraints) \cite{mokhtari2019}, convergent in the stochastic case if the mapping is strongly monotone \cite{hsieh2019,mokhtari2020}. 

In light of the above considerations, our main contributions in this paper are summarized next:
\begin{itemize}
\item In the context of (non-strictly/strongly monotone, non-cocoercive) monotone stochastic generalized Nash equilibrium problems, we propose the first distributed algorithm with a single proximal computation (e.g., projection) and a single stochastic approximation of the pseudogradient per iteration (Section \ref{sec_algo_sgnep}).
\item We show that our algorithm converges almost surely to a stochastic generalized Nash equilibrium under monotonicity of the pseudogradient with the SA scheme and the variance reduction (Section \ref{sec_conv_vr}).
\item For the stochastic non-generalized Nash equilibrium problem, we show convergence with and without the variance reduction and under several variants of monotonicity (Section \ref{sec_snep}). 
\end{itemize}

We emphasize that, unlike \cite{koshal2013,yu2017,franci2019}, we do not assume that the pseudogradient mapping is strictly/strongly monotone nor cocoercive or similar.


\section{Notation and preliminaries}

\subsection{Notation} Let $\RR$ indicate the set of real numbers and let $\bar\RR=\RR\cup\{+\infty\}$.
$\langle\cdot,\cdot\rangle:\RR^n\times\RR^n\to\RR$ denotes the standard inner product and $\norm{\cdot}$ represents the associated euclidean norm. We indicate that a matrix $A$ is positive definite, i.e., $x^\top Ax>0$, with $A\succ0$. Given a symmetric $\Phi\succ0$, denote the $\Phi$-induced inner product, $\langle x, y\rangle_{\Phi}=\langle \Phi x, y\rangle$. The associated $\Phi$-induced norm, $\norm{\cdot}_{\Phi}$, is defined as $\norm{x}_{\Phi}=\sqrt{\langle \Phi x, x\rangle}$. 
$A\otimes B$ indicates the Kronecker product between matrices $A$ and $B$. ${\bf{0}}_m$ indicates the vector with $m$ entries all equal to $0$. Given $N$ vectors $x_{1}, \ldots, x_{N} \in \RR^{n}$, $\boldsymbol{x} :=\op{col}\left(x_{1}, \dots, x_{N}\right)=\left[x_{1}^{\top}, \dots, x_{N}^{\top}\right]^{\top}.$

$J_F=(\op{Id}+F)^{-1}$ is the resolvent of the operator $F:\RR^n\to\RR^n$ and $\op{Id}$ indicates the identity operator. The set of fixed points of the operator $F$ is $\op{fix}F=\{x\in\RR^n:x=F(x)\}$. 
For a closed set $C \subseteq \RR^{n},$ the mapping $\op{proj}_{C} : \RR^{n} \to C$ denotes the projection onto $C$, i.e., $\op{proj}_{C}(x)=\op{argmin}_{y \in C}\|y-x\|$. The residual mapping is, in general, defined as $\op{res}(x^k)=\norm{x^k-\op{proj}_{C}(x^k-F(x^k))}.$ Let $g$ be a proper, lower semi-continuous, convex function. We denote the subdifferential as the maximal monotone operator $\partial g(x)=\{u\in\Omega \mid (\forall y\in\Omega)\langle y-x,u\rangle+g(x)\leq g(y)\}$. The proximal operator is defined as $\prox_{g}(v)=\op{argmin}_{u\in\Omega}\{g(u)+\tfrac{1}{2}\norm{u-v}^{2}_{}\}=J_{\partial g}(v)$.
$\iota_C$ is the indicator function of the set C, that is, $\iota_C(x)=1$ if $x\in C$ and $\iota_C(x)=0$ otherwise. The set-valued mapping $\mathrm{N}_{C} : \RR^{n} \to \RR^{n}$ denotes the normal cone operator for the the set $C$ , i.e., $\mathrm{N}_{C}(x)=\varnothing$ if $x \notin C,\left\{v \in \RR^{n} | \sup _{z \in C} v^{\top}(z-x) \leq 0\right\}$ otherwise.
Given two sets $A$ and $B$, with a slight abuse of notation, we indicate with $\op{col}(A,B)$ or $\left[\begin{smallmatrix}A\\B\end{smallmatrix}\right]$ the Cartesian product $A\times B$. This notation is common in (S)GNEPs \cite{yi2019}.

\subsection{Operator theory}
Let us collect some notions on properties of operators. The definitions are taken from \cite{facchinei2007}. First, we recall that $F$ is $\ell$-Lipschitz continuous if, for $\ell>0$,
$\norm{F(x)-F(y)} \leq \ell\norm{x-y} \text { for all } x, y \in \op{dom}(F).$
\begin{definition}[Monotone operators]
Given a mapping $F:\op{dom}(F)\subseteq\RR^n\to\RR^n$, we say that: $F$ is (strictly) monotone if for all $x, y \in \op{dom}(F)$ $(x\neq y)$
$\langle F(x)-F(y),x-y\rangle (>) \geq0;$ $F$ is (strictly) pseudomonotone if for all $x, y \in \op{dom}(F)$ $(x\neq y)$
$\langle F(y),x-y\rangle\geq 0 \Rightarrow \langle F(x),x-y\rangle(>)\geq 0;$ $\beta$-cocoercive with $\beta>0$, if for all $x, y \in \op{dom}(F)$
$\langle F(x)-F(y),x-y\rangle \geq \beta\|F(x)-F(y)\|^{2};$ $F$ is firmly nonexpansive if for all $x, y \in \op{dom}(F)$
$\|F(x)-F(y)\|^{2} \leq\|x-y\|^{2}-\|(\mathrm{Id}-F) (x)-(\mathrm{Id}-F) (y)\|^{2}.$
\end{definition}
An example of firmly nonexpansive operator is the projection operator over a nonempty, compact and convex set \cite[Proposition 4.16]{bau2011}. We note that a firmly nonexpansive operator is also nonexpansive and firmly quasinonexpansive \cite[Definition 4.1]{bau2011}.
We note that if a mapping is $\beta$-cocoercive it is also $1/\beta$-Lipschitz continuous \cite[Remark 4.15]{bau2011}.

\section{Stochastic Generalized Nash equilibrium problems}\label{sec_GNEPs}
We consider a set of noncooperative agents $\mc I=\{1,\dots,N\}$, each of them choosing its strategy $x_i\in\RR^{n_i}$ with the aim of minimizing its local cost function within its feasible strategy set. 
The local decision set of each agent is indicated with $\Omega_i$, i.e., for all $i\in\mc I$, $x_i\in\Omega_i\subseteq\RR^{n_i}$.
Besides the local set, each agent is subject to some joint feasibility constraints, $g(x)\leq 0$. Let us set $n=\sum_{i=1}^Nn_i$ and $\bs\Omega=\prod_{i=1}^N\Omega_i$, then, the collective feasible set can be written as
\begin{equation}\label{collective_set}
\bs{\mc{X}}=\bs\Omega\cap\left\{\bs y \in\bs\RR^n\; | \;g(\bs y) \leq {\bf{0}}_{m}\right\},
\end{equation}
where $g:\RR^{nN}\to\RR^m$ \cite{belgioioso2017}. Let us also indicate with $\mc X_i(\bs x_{-i})$ the piece of coupling constraints corresponding to agent $i$, which is affected by the decision variables of the other agents $\bs x_{-i}=\op{col}((x_j)_{j\neq i})$.

\begin{assumption}[Constraint qualification]\label{ass_constr}
For each $i \in \mc I,$ the set $\Omega_{i}$ is nonempty, closed and convex.
The set $\bs{\mc{X}}$ satisfies Slater's constraint qualification. 
\fineass\end{assumption}
\begin{assumption}[Separable convex coupling constraints] 
The mapping $g$ in \eqref{collective_set} has a separable form, $g(\boldsymbol{x}):=\sum_{i=1}^{N} g_{i}\left(x^{i}\right)$, for some convex differentiable functions $g_{i}:\mathbb{R}^{n} \rightarrow \mathbb{R}^{m}, i\in\mc I$ and it is $\ell_{\mathrm{g}}$-Lipschitz continuous. Its gradient $\nabla g$ is bounded, i.e., $\op{sup}_{\bs x\in\bs{\mc X}}\|\nabla g(\boldsymbol{x})\| \leq B_{\nabla \mathrm{g}}$.
\end{assumption}

The local cost function of agent $i$ is defined as 
\begin{equation}\label{eq_cost_stoc}
\JJ_i(x_i,\bs{x}_{-i})=\EE_\xi[J_i(x_i,\bs{x}_{-i},\xi(\varpi))]+f_i(x_i),
\end{equation}
for some measurable function $J_i:\mc \RR^{n}\times \RR^d\to \RR$. The cost function $\JJ_i$ of agent $i\in\mc I$ depends on the local variable $x_i$, the decisions of the other players $\bs x_{-i}$ and the random variable $\xi:\Xi\to\RR^d$ that express the uncertainty.
$\EE_\xi$ represent the mathematical expectation with respect to the distribution of the random variable $\xi(\varpi)$\footnote{From now on, we use $\xi$ instead of $\xi(\varpi)$ and $\EE$ instead of $\EE_\xi$.} in the probability space $(\Xi, \mc F, \PP)$. We assume that $\EE[J_i(\bs{x},\xi)]$ is well defined for all feasible $\bs{x}\in\bs{\mc X}$ \cite{ravat2011}. Moreover, the cost function presents the typical splitting in a smooth part and a nonsmooth part. The latter is indicated with $f_i:\RR^{n_i}\to\bar{\RR}$ and it can represent a local cost or local constraints via the indicator function, i.e. $f_i(x_i)=\iota_{\Omega_i}(x_i)$.

\begin{assumption}[Cost function convexity]\label{ass_J}
For each $i\in\mc I$, the function $f_i$ in \eqref{eq_cost_stoc} is lower semicontinuous and convex and $\op{dom}(f_i)=\Omega_i$. 
For each $i \in \mc I$ and $\boldsymbol{x}_{-i} \in \bs{\mc{X}}_{-i}$ the function $\JJ_{i}(\cdot, \boldsymbol{x}_{-i})$ is convex and continuously differentiable.
\fineass\end{assumption}

Given the decision variables of the other agents $\bs{x}_{-i}$, each agent $i$ aims at choosing a strategy $x_i$, that solves its local optimization problem, i.e.,
\begin{equation}\label{eq_game}
\forall i \in \mc I: \quad\left\{\begin{array}{cl}
\min_{x_i \in \Omega_i} & \JJ_i\left(x_i, \bs{x}_{-i}\right)\\ 
\text { s.t. } & g(x_i,\bs x_{-i})\leq0.
\end{array} \right.
\end{equation}
From a game-theoretic perspective, the solution concept that we are seeking is that of stochastic generalized Nash equilibrium (SGNE).
\begin{definition}\label{def_GNE}
A Stochastic Generalized Nash equilibrium is a collective strategy $\bs x^*\in\bs{\mc X}$ such that for all $i \in \mc I$
$$\JJ_i(x_i^{*}, \boldsymbol x_{-i}^{*}) \leq \inf \{\JJ_i(y, \boldsymbol x_{-i}^{*})\; | \; y \in \mc{X}_i(\bs x_{-i})\}.$$
\end{definition}

In other words, a SGNE is a set of strategies where no agent can decrease its objective function by unilaterally deviating from its decision.
To guarantee that a SGNE exists, we make further assumptions on the cost functions \cite{ravat2011}.
\begin{assumption}[Convexity and measurability]\label{ass_J_exp}
For each $i\in\mc I$ and for each $\xi \in \Xi$, the function $J_{i}(\cdot,\boldsymbol x_{-i},\xi)$ is convex, Lipschitz continuous, and continuously differentiable. The function $J_{i}(x_i,\bs x_{-i},\cdot)$ is measurable and for each $\boldsymbol x_{-i}$, the Lipschitz constant $\ell_i(\boldsymbol x_{-i},\xi)$ is integrable in $\xi$.
\fineass\end{assumption}

Existence of a SGNE of the game in \eqref{eq_game} is guaranteed, under Assumptions \ref{ass_constr}--\ref{ass_J_exp}, by \cite[Section 3.1]{ravat2011} while uniqueness does not hold in general \cite[Section 3.2]{ravat2011}.
Within all the possible Nash equilibria, we focus on those that corresponds to the solution set of an appropriate stochastic variational inequality. 
To this aim, let us denote the pseudogradient mapping as
\begin{equation}\label{eq_grad}
\FF(\bs{x})=\op{col}\left(\EE[\nabla_{x_{1}} J_{1}(x_{1}, \bs{x}_{-1})], \dots, \EE[\nabla_{x_{N}} J_{N}(x_{N}, \bs{x}_{-N})]\right)
\end{equation} 
and let
$\partial f(\bs x)=\op{col}(\partial f_1(x_1),\dots,\partial f_N(x_N)).$
The possibility to exchange the expected value and the pseudogradient $\FF$ in \eqref{eq_grad} is guaranteed by Assumption \ref{ass_J_exp}. Then, the associated stochastic variational inequality (SVI) reads as
\begin{equation}\label{eq_SVI}
\langle \FF(\bs x^*),\bs x-\bs x^*\rangle+\sum_{i\in\mc I}\left\{f_i(x_i)-f_i(x^*_i)\right\}
\geq 0,\text { for all } \bs x \in \bs{\mc X},
\end{equation}
where $\bs{\mc X}$ is the intersection of the local and coupling constraints as in \eqref{collective_set}.
When Assumptions \ref{ass_constr}--\ref{ass_J_exp} hold, any solution of $\op{SVI}(\bs{\mc X} , \FF)$ in \eqref{eq_SVI} is a SGNE of the game in (\ref{eq_game}) while vice versa does not hold in general. This is because a game may have a Nash equilibrium while the corresponding VI may have no solution \cite[Proposition 12.7]{palomar2010}. 
\begin{assumption}[Existence of a variational equilibrium]\label{ass_sol}
The SVI in \eqref{eq_SVI} has at least one solution, i.e., $\op{SOL}(\bs{\mc X},\FF)\neq\varnothing$.\fineass
\end{assumption}
\begin{remark}\label{remark_set}
Assumption \ref{ass_sol} is satisfied if the sets $\Omega_i$, $i\in\mc I$, are compact \cite[Corollary 2.2.5]{facchinei2007}.\fineass
\end{remark}
We call variational equilibria (v-SGNE) the SGNE that are also solution of the associated SVI, namely, the solution of the $\op{SVI}(\bs{\mc X} , \FF)$ in (\ref{eq_SVI}) with $\FF$ in (\ref{eq_grad}) and $\bs{\mc X}$ in (\ref{collective_set}). 


In the remaining part of this section, we recast the SGNEP as a monotone inclusion, i.e., the problem of finding a zero of a set-valued monotone operator. To this aim, we characterize the SGNE of the game in terms of the Karush--Kuhn--Tucker (KKT) conditions of the coupled optimization problems in (\ref{eq_game}). 
Let us define the Lagrangian function, for each $i\in\mc I$, with
$$\mc L_i\left(\bs{x}, \lambda_i\right) :=\JJ_i\left(x_i, \bs{x}_{-i}\right)+f_i\left(x_i\right)+\lambda_i^{\top}g(x_i^*,\bs x_{-i}^*),$$
where $\lambda_i \in \RR_{ \geq 0}^{m}$ is the Lagrangian dual variable associated with the coupling constraints. 
Then, a set of strategies $\bs x^{*}$ is a SGNE if and only if the KKT conditions are satisfied \cite[Theorem 4.6]{facchinei2010}.
Moreover, according to \cite[Theorem 3.1]{facchinei2007vi}, \cite[Theorem 3.1]{auslender2000}, the v-SGNE are those equilibria such that the shared constraints have the same dual variable for all the agents, i.e. $\lambda_i=\lambda$ for all $i\in\mc I$, and solve the $\op{SVI}(\bs{\mc X},\FF)$ in \eqref{eq_SVI}. Thus, $\bs x^*$ is a v-SGNE if the following KKT inclusions, for all $i\in\mc I$, are satisfied for some $\lambda\in\RR^m_{\geq0}$:
\begin{equation}\label{eq_KKT_VI}
\begin{cases}
0 \in \EE[\nabla_{x_i} J_i(x_i^{*}, \bs{x}_{-i}^{*},\xi)]+\partial f_i\left(x_i^{*}\right)+\nabla g(x_i^*,\bs x_{-i}^*)^\top \lambda\\ 
0\in -g(\bs x^*)+N_{\RR^m_{\geq 0}}(\lambda).
\end{cases}
\end{equation}

\section{Distributed stochastic relaxed forward--backward algorithm}\label{sec_algo_sgnep}
In this section we describe the details that lead to the distributed iterations in Algorithm \ref{algo_i} which include an averaging step \eqref{eq_ave} and a proximal step \eqref{eq_prox}. The averaging step induces some inertia but it allows us to prove convergence under mild monotonicity assumptions. Moreover, for the decision variable $x_i$, the proximal update guarantees that the local constraints are always satisfied while the coupling constraints are enforced asymptotically through the dual variable $\lambda_i$, which should be nonnegative. The variable $z_i$ is an auxiliary variable to force consensus on the dual variables. We note that to update the primal variable we use an approximation $\hat F$ of the pseudogradient mapping $\FF$, characterized in Section \ref{sec_approx}.
\begin{algorithm}[t]
\caption{Stochastic Relaxed Forward Backward (SRFB)}\label{algo_i}
Initialization: $x_i^0 \in \Omega_i, \lambda_i^0 \in \RR_{\geq0}^{m},$ and $z_i^0 \in \RR^{m} .$\\
Iteration $k$: Agent $i$\\
(1) Updates the variables
\begin{equation}\label{eq_ave}
\begin{aligned}
&\bar x_{i}^{k}=(1-\delta)x_i^k+\delta\bar x_i^{k-1}\\
&\bar z_{i}^{k}=(1-\delta)z_i^k+\delta\bar z_i^{k-1}\\
&\bar \lambda_{i}^{k}=(1-\delta)\lambda_i^k+\delta\bar \lambda_i^{k-1}\\
\end{aligned}
\end{equation}
(2) Receives $x_{j}^k$ for all $j \in \mathcal{N}_{i}^{J}$ and $z_j^k, \lambda_{j}^k$ for $j \in \mathcal{N}_{i}^{\lambda}$, then updates:
\begin{equation}\label{eq_prox}
\begin{aligned}
x_i^{k+1}&=\op{prox}_{f_i}[\bar x_i^k-\alpha_{i}(\hat F_{i}(x_i^k, \boldsymbol{x}_{-i}^k,\xi_i^k)+\nabla g_i(x_i)^\top \lambda_i^k)]\\
z_i^{k+1}&=\bar z_i^k-\nu_{i}\textstyle{ \sum_{j \in \mathcal{N}_{i}^{\lambda}}} w_{i,j}(\lambda_i^k-\lambda_{j}^k)\\
\lambda_i^{k+1}&=\op{proj}_{\RR^m_{\geq 0}}\{\bar \lambda_i^k+\tau_{i}g_i(x_i^k)\\
&-\tau\textstyle{\sum_{j \in \mathcal{N}_{i}^{\lambda}}} w_{i,j}[(z_{i}^{k}-z_j^k)-(\lambda_i^k-\lambda_j^k)]\}
\end{aligned}
\end{equation}
\end{algorithm}
We suppose that each player $i$ knows its local data $\Omega_i$, and their part $\mc X_i(\bs x_{-i})$.
We also suppose that the agents have access to a pool of samples of the random variable and are able to compute, given the actions of the other players $\bs x_{-i}$, the pseudogradient $\FF$ of their own cost functions (or an approximation $\hat F$). 
The set of agents $j$ whose decisions affect the cost function of agent $i$, are denoted by $\mc N_i^J$. Specifically, $j\in\mc N_i^J$ if the function $J_i(x_i,\bs x_{-i})$ explicitly depends on $x_j$.
Under these premises, Algorithm \ref{algo_i} is distributed in the sense that each agent knows its own problem data and variables and communicates with the other agents only to access the information to compute $\FF$ (full-decision information setup \cite{yi2019}).

Let us also introduce the graph $\mc G^\lambda=(\mc I,\mc E^\lambda)$ through which a local copy of the dual variable is shared. According to \cite[Theorem 3.1]{facchinei2007vi}, \cite[Theorem 3.1]{auslender2000}, we seek for a v-SGNE with consensus of the dual variables. Therefore, along with the dual variable, agents share through $\mc G^\lambda$ a copy of an auxiliary variable $z_i\in\RR^m$ whose role is to force consensus. A deeper insight on this variable is given later in this section.
The set of edges $\mc E^\lambda$ of the multiplier graph $\mc G^\lambda$, is given by: $(i,j)\in\mc E^\lambda$ if player $j$ share its $\{\lambda_j,z_j\}$ with player $i$. For all $i\in\mc I$, the neighboring agents in $\mc G^\lambda$ form the set $\mc N^\lambda_i=\{j\in\mc I:(i,j)\in\mc E^\lambda\}$. In this way, each agent controls his own decision variable and a local copy of the dual variable $\lambda_i$ and of the auxiliary variable $z_i$ and, through the graphs, it obtains the other agents variables. 
\begin{assumption}[Graph connectivity]\label{ass_Jraph}
The multiplier graph $\mc G^\lambda$ is undirected and connected.
\fineass\end{assumption}
The weighted adjacency matrix associated to $\mc G^\lambda$ is denoted with $W\in\RR^{N\times N}$. Then, letting $D=\op{diag}\{d_1,\dots,d_N\}$ where $d_i=\sum_{j=1}^Nw_{i,j}$ is the degree of agent $i$, the associated Laplacian is given by $L=D-W\in\RR^{N\times N}$. It follows from Assumption \ref{ass_Jraph} that $L=L^\top$.

Let us now rewrite the KKT conditions in \eqref{eq_KKT_VI} in compact form as
\begin{equation}\label{eq_T}
0\in\mc T(\bs{x},\bs\lambda)=\left[\begin{array}{c}
\FF(\bs{x})+\partial f(\bs x)+\nabla g(\bs x)^{\top}\bs \lambda \\ 
\mathrm{N}_{\RR_{ \geq 0}^{m}}(\bs\lambda)-g(\bs x)
\end{array}\right].
\end{equation}
$\mc T:\mc X\times \RR^m_{\geq 0}\rightrightarrows \RR^{n}\times\RR^m$ is a set-valued mapping and it follows that the v-SGNE of the game in \eqref{eq_game} correspond to the zeros of the mapping $\mc T$ which can be split as a summation of two operators, $\mc T=\mc A+\mc B$, where
\begin{equation}\label{eq_splitting}
\begin{aligned}
\mc{A} &:\left[\begin{array}{l}
\bs{x} \\
\lambda
\end{array}\right] \mapsto\left[\begin{array}{c}
\FF(\bs{x}) \\ 
0
\end{array}\right]+\left[\begin{array}{c}
\nabla g(\bs x)^\top\bs \lambda\\
-g(\bs x)
\end{array}\right]\\
\mc{B} &:\left[\begin{array}{l}
\bs{x} \\ 
\lambda
\end{array}\right] \mapsto\left[\begin{array}{c}
\partial f(\bs x) \\ 
\mathrm{N}_{\RR_{ \geq 0}^{m}}(\lambda)
\end{array}\right].
\end{aligned}
\end{equation}

To force consensus on the dual variables, the authors in \cite{yi2019} proposed the Laplacian constraint ${\bf{L}}\bs\lambda=0$. This is why, to preserve monotonicity, we expand the two operators $\mc A$ and $\mc B$ in \eqref{eq_splitting} and introduce the auxiliary variable $\bs z=\op{col}(z_1,\dots,z_N)\in\RR^{Nm}$.
Let us first define ${\bf{L}}=L\otimes \op{Id}_m\in\RR^{Nm\times Nm}$ where $L$ is the laplacian of $\mc G^\lambda$ and $\bs \lambda=\op{col}(\lambda_1,\dots,\lambda_N)\in\RR^{Nm}$. Then, the two operators $\mc A$ and $\mc B$ in \eqref{eq_splitting} can be rewritten as
\begin{equation}\label{eq_expanded}
\begin{aligned}
\bar{\mc{A}} &:\left[\begin{array}{l}
\bs{x} \\
\bs z\\
\bs \lambda
\end{array}\right] \mapsto\left[\begin{array}{c}
\FF(\bs{x}) \\ 
0\\
{\bf{L}}\bs\lambda
\end{array}\right]+
\left[\begin{array}{ccc}
\nabla g(\bs x)^\top\bs\lambda\\
{\bf{L}} \bs z\\
-g(\bs x)-{\bf{L}} \bs z
\end{array}\right]\\
\bar{\mc{B}} &:\left[\begin{array}{l}
\bs{x} \\ 
\bs z\\
\bs \lambda
\end{array}\right] \mapsto\left[\begin{array}{c}
\partial f(\bs x) \\ 
{\bf{0}}\\
\mathrm{N}_{\RR_{ \geq 0}^{m}}(\lambda)
\end{array}\right].
\end{aligned}
\end{equation}

From now on, we indicate the state variable as $\bs{\omega}=\op{col}(\bs x,\bs z,\bs \lambda)$. The properties of the operators in \eqref{eq_expanded} depends on the properties of $\FF$ and are described in the next section. We here show that the zeros of $\bar{\mc A}+\bar{\mc B}$ are the same as the zeros of $\mc T$ in \eqref{eq_T}.

\begin{lemma}\label{lemma_zero}
Let Assumptions \ref{ass_constr}--\ref{ass_Jraph} hold and consider the operators $\mc A$ and $\mc B$ in (\ref{eq_splitting}), and the operators $\bar{\mc A}$ and $\bar{\mc B}$ in (\ref{eq_expanded}). Then the following statements hold.
\begin{enumerate}
\item[(i)] Given any $\bs\omega^* \in \op{zer}(\bar{\mc A}+\bar{\mc B})$, $\bs{x}^{*}$ is a v-SGNE of game in (\ref{eq_game}), i.e., $\bs{x}^*$ solves the $\op{SVI}(\bs{\mc X} , \FF)$ in (\ref{eq_SVI}). Moreover $\bs{\lambda}^{*}=\mathbf{1}_{N} \otimes \lambda^{*},$ and $(\bs{x}^*,\lambda^{*})$ satisfy the KKT condition in (\ref{eq_KKT_VI}) i.e., $\op{col}(\bs{x}^*, \lambda^{*}) \in \op{zer}(\mc A+\mc B)$.
\item[(ii)] $\op{zer}(\mc A+\mc B) \neq \emptyset$ and $\op{zer}(\bar{\mc A}+\bar{\mc B}) \neq \emptyset$.
\end{enumerate}
\end{lemma}
\begin{proof}
See Appendix \ref{app_op}.
\end{proof}

Since the distribution of the random variable is unknown, the expected value mapping can be hard to compute. Therefore, we take an approximation of the pseudogradient mapping, properly defined in Section \ref{sec_approx}. Therefore, in what follows, we replace $\bar{\mc A}$ with 
\begin{equation}\label{eq_A_hat}
\hat{\mc A}:\left[\begin{array}{c}\hspace{-.1cm}
(\bs{x},\xi)\hspace{-.1cm} \\ 
\bs z\\
\bs \lambda
\end{array}\right] \hspace{-.1cm}\mapsto\hspace{-.1cm}\left[\begin{array}{c}
\hat F(\bs{x},\xi) \\ 
0\\
\hspace{-.1cm}\bf L\bs\lambda\hspace{-.1cm}
\end{array}\right]+\left[\begin{array}{ccc}
\nabla g(\bs x)^\top\bs\lambda\\
{\bf{L}} \bs z\\
-g(\bs x)-{\bf{L}} \bs z
\end{array}\right].
\end{equation}
where $\hat F$ is an approximation of the expected value mapping $\FF$ in (\ref{eq_grad}) given some realizations of the random vector $\xi$.

Then, Algorithm \ref{algo_i} can be written in compact form as \cite{malitsky2019}
\begin{equation}\label{algo}
\begin{aligned} 
\bar{\bs \omega}^{k} &=(1-\delta)\bs \omega^k+\delta\bar{\bs \omega}^{k-1} \\ 
\bs \omega^{k+1} &=(\op{Id}+\Phi^{-1}\bar{\mc B})^{-1}(\bar{\bs \omega}^{k}-\Phi^{-1}\hat{\mc A}(\bs \omega^k)).
\end{aligned}
\end{equation}
where $\Phi\succ0$ contains the inverse of step size sequences 
\begin{equation}\label{eq_phi}
\Phi=\op{diag}(\alpha^{-1},\nu^{-1},\sigma^{-1}),
\end{equation}
and $\alpha^{-1}$, $\nu^{-1}$, $\sigma^{-1}$ are diagonal matrices.

\subsection{Approximation Scheme}\label{sec_approx}

We now enter the details of the approximation introduced in Algorithm \ref{algo_i}. We use a stochastic approximation (SA) scheme with variance reduction, hence, we suppose that, at each iteration $k$, the agents have access to a pool of samples of the random variable and are able to compute an approximation of $\FF(\bs x)$ of the form
\begin{equation}\label{eq_F_SAA}
\begin{aligned}
&\hat F(\bs x,\bs\xi)=\op{col}(\hat F_i(\bs x,\bs\xi))\\
&=\op{col}\left(\frac{1}{S_k} \sum_{t=1}^{S_k} \nabla_{x_1}J_1(\bs x,\xi_1^{(t)}),\dots,\frac{1}{S_k} \sum_{t=1}^{S_k} \nabla_{x_N}J_N(\bs x, \xi_N^{(t)})\right).
\end{aligned}
\end{equation}
where $\bs \xi=\op{col}(\bar \xi_1,\dots,\bar\xi_N)$, for all $i\in\mc I$, $\bar\xi_i=\op{col}(\xi_i^{(1)},\dots,\xi_i^{(S_k)})$ and $\bs \xi$ is an i.i.d. sequence of random variables drawn from $\PP$. Approximations of the form (\ref{eq_F_SAA}) are very common in Monte-Carlo simulation approaches, machine learning \cite{staudigl2019}; they are easy to obtain in case we are able to sample from the measure $\PP$.
Typical assumptions when using an approximation as in \eqref{eq_F_SAA} are related to the choice of a proper batch size sequence $S_k$ \cite{staudigl2019,iusem2017}.

\begin{assumption}[Increasing batch size]\label{ass_batch}
The batch size sequence $(S_k)_{k\geq 1}$ is such that, for some $c,k_0,a>0$,
$$S_k\geq c(k+k_0)^{a+1}.\vspace{-.65cm}$$

\fineass\end{assumption}
Form Assumption \ref{ass_batch}, it follows that $1/S_k$ is summable, which is a standard assumption when used in combination with the forthcoming variance reduction assumption to control the stochastic error \cite{staudigl2019,iusem2017}. 
For $k \geq 0,$ the approximation error is defined as
$$\epsilon_k= \hat F(\bs x^k,\xi^k)-\FF(\bs x^k).$$

\begin{remark}\label{remark_error}
Since there is no uncertainty in the constraints, we indicate with $\varepsilon_k=\op{col}(\epsilon_k,0,0)$ the error on the extended operator, i.e.,
$\hat{\mc A}(\bs\omega^k,\xi^k)-\mc A(\bs\omega^k)=\varepsilon_k.$
$\hat{\mc A}$ is the operator in \eqref{eq_A_hat} with approximation $\hat F$ in \eqref{eq_F_SAA}.
\fineass\end{remark}
In the stochastic framework, there are usually assumptions on the expected value and variance of the stochastic error $\epsilon_k$ \cite{iusem2017,koshal2013,kannan2019}.
Let us define the filtration $\mc F=\{\mc F_k\}$, that is, a family of $\sigma$-algebras such that $\mathcal{F}_{0} = \sigma\left(X_{0}\right)$ and $\mathcal{F}_{k} = \sigma\left(X_{0}, \xi_{1}, \xi_{2}, \ldots, \xi_{k}\right) \quad \forall k \geq 1,$
such that $\mc F_k\subseteq\mc F_{k+1}$ for all $k\in\NN$. In words, $\mc F_k$ contains the information up to time $k$.

\begin{assumption}[Zero mean error]\label{ass_error}
The stochastic error is such that, for all $k\in\NN$, a.s.,
$\EEk{\epsilon^k}=0.$
\fineass\end{assumption}

\begin{assumption}[Variance control]\label{ass_variance}
There exist $p \geq 2$, $\sigma_{0} \geq 0$ and a measurable locally bounded function $\sigma : \op{SOL}(\bs{\mc X},\FF) \rightarrow \mathbb{R} $ such that for all $(\bs x,\bs x^* \in \bs{\mc X} \times \op{SOL}(\bs{\mc X},\FF)$ 
\begin{equation*}\label{eq_variance}
\mathbb{E}[\|\hat F(\bs x, \xi)-\FF(\bs x)\|^{p}]^{1/p} \leq \sigma\left(\bs x^{*}\right)+\sigma_{0}\left\|\bs x-\bs x^{*}\right\|.\vspace{-.5cm}
\end{equation*}
\fineass\end{assumption}
\begin{remark}
When the feasible set is compact (as in Remark \ref{remark_set}) an uniform bounded variance, i.e., for some $\sigma>0$
\begin{equation*}
\op{\sup}_{\bs x\in\bs{\mc X}}\EE[\|\hat F(\bs x, \xi)-\FF(\bs x)\|^{2}]\leq \sigma^2,
\end{equation*}
can be considered instead of Assumption \ref{ass_variance}.
\fineass\end{remark}

\subsection{Convergence analysis}\label{sec_conv_vr}

Now, we study the convergence of the algorithm.
First, to ensure that $\bar{\mc A}$ and $\bar{\mc B}$ have the properties that we use for the analysis, we make the following assumption.
\begin{assumption}[Monotonicity]\label{ass_mono}
$\FF$ in \eqref{eq_grad} is monotone and $\ell_\FF$-Lipschitz continuous for some $\ell_\FF>0$.
\fineass\end{assumption}

\begin{lemma}\label{lemma_op}
Let Assumptions \ref{ass_Jraph} and \ref{ass_mono} hold and let $\Phi\succ 0$. Then, the operators $\bar{\mc A}$ and $\bar{\mc B}$ in \eqref{eq_expanded} have the following properties.
\begin{enumerate}
\item $\bar{\mc A}$ is monotone and $\ell_{\bar{\mc A}}$-Lipschitz continuous.
\item The operator $\bar{\mc B}$ is maximally monotone.
\item $\Phi^{-1}\bar{\mc A}$ is monotone and $\ell_{\Phi}$-Lipschitz continuous.
\item $\Phi^{-1}\bar{\mc B}$ is maximally monotone.
\end{enumerate}
\end{lemma}
\begin{proof}
See Appendix \ref{app_op}.
\end{proof}

Lastly, we indicate how to choose the parameters of the algorithm. This is fundamental for the convergence analysis and, in practice, for the convergence speed.

\begin{assumption}[Averaging parameter]\label{ass_delta}
The averaging parameter $\delta$ in \eqref{algo} is such that
$$\textstyle{\frac{1}{\varphi}\leq\delta\leq 1}$$
where $\varphi=\frac{1+\sqrt{5}}{2}$ is the golden ratio. 
\fineass\end{assumption}
\begin{assumption}[Step size bound]\label{ass_step1}
The steps size is such that
$$0<\norm{\Phi^{-1}}\leq \tfrac{1}{2\delta(2\ell_{\bar{\mc A}}+1)}$$
where $\ell_{\bar{\mc A}}$ is the Lipschitz constant of $\bar{\mc A}$ as in Lemma \ref{lemma_op}.
\fineass\end{assumption}

We are now ready to state our convergence result.

\begin{theorem}\label{theo_sgne}
Let Assumptions \ref{ass_constr}--\ref{ass_step1} hold. Then, the sequence $(\bs x^k)_{k\in\NN}$ generated by Algorithm \ref{algo_i} with $\hat F$ as in \eqref{eq_F_SAA} converges a.s. to a v-SGNE of the game in \eqref{eq_game}.  
\end{theorem}
\begin{proof}
See Appendix \ref{app_sec_algo}.
\end{proof}

In the following, let us consider the case where the local nonsmooth cost is determined by the local constraints, i.e., $f_i(x_i)=\iota_{\Omega_i}(x_i)$. Then, the problem is slightly different and we can show that the algorithm converges under a weaker assumption than monotonicity.

The first difference is that the operator $\bar{\mc B}$ is now given by 
\begin{equation*}
\bar{\mc{B}} :\left[\begin{array}{l}
\bs{x} \\ 
\bs z\\
\bs \lambda
\end{array}\right] \rightrightarrows\left[\begin{array}{c}
\mathrm N_{\bs\Omega}(\bs x) \\ 
{\bf{0}}\\
\mathrm{N}_{\RR_{ \geq 0}^{m}}(\lambda)
\end{array}\right],
\end{equation*}
hence, we have a projection instead of the proximal operator:
\begin{equation}\label{algo_proj}
\begin{aligned} 
\bar{\bs \omega}^{k} &=(1-\delta)\bs \omega^k+\delta\bar{\bs \omega}^{k-1} \\ 
\bs \omega^{k+1} &=\op{proj}_{\bs{\mc Z}'}\left(\bar{\bs \omega}^{k}-\Phi^{-1}\hat{\mc A}(\bs \omega^k)\right),
\end{aligned}
\end{equation}
where $\bs{\mc Z}'=\bs\Omega\times\RR^{mN}\times\RR^{mN}_{\geq0}$. We also call $\bs{\mc Z}=\bs{\mc X}\times\RR^{mN}\times\RR^{mN}_{\geq0}$ and $\bs{\mc Z}^*$ the set of v-SGNE, i.e., $\bs{\mc Z}^*=\op{zer}(\bar{\mc A}+\bar{\mc B})$.
To show convergence, let the mapping $\bar{\mc A}$ satisfy the following assumption.
\begin{assumption}[Almost restricted pseudo monotonicity]\label{ass_weaker}
The operator $\bar{\mc A}$ in \eqref{eq_expanded} is such that for all $(\bs\omega,\bs \omega^*)\in\bs{\mc Z}\times\bs{\mc Z}^*$
\begin{equation*}\label{eq_weak}
\langle \bar{\mc A}(\bs \omega),\bs \omega-\bs\omega^*\rangle\geq 0.\vspace{-.55cm}
\end{equation*}\fineass
\end{assumption}

\begin{remark}\label{remark_weak}
The property in Assumption \ref{ass_weaker} is implied by both monotonicity and pseudomonotonicity but it does not necessarily hold for $\bar{\mc A}$ if we assume it directly on $\FF$. It corresponds to the concept of weak solution of a VI, compared to that of strong solution as in \eqref{eq_SVI} \cite{dang2015}. This assumption is also used in \cite{staudigl2019} and \cite{solodov1999}; an example of a mapping that satisfies (\ref{eq_weak}) is in \cite[Equation 2.4]{dang2015}.
\fineass\end{remark}

We can now state the corresponding convergence result.

\begin{corollary}\label{cor_coco}
Let Assumptions \ref{ass_constr}--\ref{ass_variance} and \ref{ass_delta}--\ref{ass_weaker} hold. Then, the sequence generated by Algorithm \ref{algo_i} in \eqref{algo_proj} with $\hat{\mc A}$ as in Remark \ref{remark_error} converges a.s. to a v-SGNE of the game in \eqref{eq_game}.
\end{corollary}
\begin{proof}
See Appendix \ref{app_sec_algo}.
\end{proof}


\section{Convergence under cocoercivity}

\begin{algorithm}[t]
\caption{Stochastic Relaxed Preconditioned Forward Backward (SRpFB)}\label{algo_i2}
Initialization: $x_i^0 \in \Omega_i, \lambda_i^0 \in \RR_{\geq0}^{m},$ and $z_i^0 \in \RR^{m} .$
Iteration $k$: Agent $i$\\
(1) Updates the variables
$$\begin{aligned}
&\bar x_{i}^{k}=(1-\delta)x_i^k+\delta\bar x_i^{k-1}\\
&\bar z_{i}^{k}=(1-\delta)z_i^k+\delta\bar z_i^{k-1}\\
&\bar \lambda_{i}^{k}=(1-\delta)\lambda_i^k+\delta\bar \lambda_i^{k-1}\\
\end{aligned}$$
(2) Receives $x_j^k$ for all $j \in \mathcal{N}_{i}^{J}, \lambda_{j}^k$ for $j \in \mathcal{N}_{i}^{\lambda}$ then updates:
$$\begin{aligned}
&x_i^{k+1}=\op{proj}_{\Omega_i}[\bar x_i^k-\alpha_{i}(\hat F_{i}(x_i^k, \boldsymbol{x}_{-i}^k,\xi_i^k)-A_{i}^\top \lambda_i^k)]\\
&z_i^{k+1}=\bar z_i^k-\nu_{i}\textstyle{ \sum_{j \in \mathcal{N}_{i}^{\lambda}}} w_{i,j}(\lambda_i^k-\lambda_{j}^k)\\
\end{aligned}$$
(3) Receives $z_{j, k+1}$ for all $j \in \mathcal{N}_{i}^{\lambda}$ then updates:
$$\begin{aligned}
&\lambda_i^{k+1}=\op{proj}_{\RR_{+}^{m}}\left[\bar\lambda_i^k+\sigma_{i}\left(A_{i}(2x_i^{k+1}-x_i^k)-b_{i}\right)\right.\\
&\qquad+\sigma_{i}\textstyle{\sum_{j \in \mathcal{N}_{i}^{\lambda}}} w_{i,j}\left(2(z_i^{k+1}-z_{j}^{k+1})-(z_i^k-z_{j}^k)\right)\\
&\qquad-\sigma_{i}\textstyle{\sum_{j \in \mathcal{N}_{i}^{\lambda}}} \left.w_{i,j}(\lambda_i^k-\lambda_{j}^k)\right]\\
\end{aligned}$$
\end{algorithm}

Having a mild monotonicity condition on the pseudogradient implies taking a small, although constant, step size sequence. However, if the pseudogradient mapping satisfies a stronger monotonicity assumption, a larger step size can be chosen. 
One possibility is that the pseudogradient satisfies the cut property (described in details in Remark \ref{remark_cut} later on), which is hard to check on the problem data but it follows directly from cocoercivity. 
\begin{assumption}[Cocoercivity]\label{ass_coco}
$\FF$ in \eqref{eq_grad} is $\beta$-cocoercive for some $\beta>0$.
\fineass\end{assumption}
For instance, every symmetric, affine, monotone mapping is cocoercive (see also \cite[Example 2.9.25]{facchinei2007}).
The operator splitting that we used in Section \ref{sec_algo_sgnep} is not cocoercive, even when the mapping $\FF$ is. For this reason, here we have to consider a different splitting. Moreover, to obtain distributed iterations, we consider affine coupling constraints.
\begin{assumption}[Affine coupling constraints]
$g(\bs x)=A\bs x -b$, where $A=[A_1,\dots,A_N]\in\RR^{m\times n}$ and $b\in\RR^m$. 
\fineass\end{assumption}
\cite[Theorem 3.1]{facchinei2007vi}, \cite[Theorem 3.1]{auslender2000} holds also in this case. To obtain the extended operator as in \eqref{eq_expanded}, let ${\bf{A}}=\op{diag}\{A_1,\dots,A_N\}\in\RR^{Nm\times n}$ where $A_i$ represents the individual coupling constraints and let $\bs b\in\RR^{Nm\times n}$. Then, the operator $\mc T$ in \eqref{eq_T} can be split as
\begin{equation}\label{eq_expanded2}
\begin{aligned}
\bar{\mc{C}} &:\left[\begin{array}{l}
\bs{x} \\
\bs z\\
\bs \lambda
\end{array}\right] \mapsto\left[\begin{array}{c}
\FF(\bs{x}) \\ 
0\\
{\bf{L}}\bs\lambda+\bs b
\end{array}\right]\\
\bar{\mc{D}} &:\left[\begin{array}{l}
\bs{x} \\ 
\bs z\\
\bs \lambda
\end{array}\right] \mapsto\left[\begin{array}{c}
\mathrm{N}_{\Omega}(\bs{x})  \\ 
{\bf{0}}\\
\mathrm{N}_{\RR_{ \geq 0}^{m}}(\lambda)
\end{array}\right]+\left[\begin{array}{ccc}
0 & 0 & {\bf{A}}^{\top} \\ 
0 & 0 & {\bf{L}}\\
-{\bf{A}} & -{\bf{L}} & 0
\end{array}\right]\left[\begin{array}{l}
\bs{x} \\ 
\bs z\\
\bs \lambda
\end{array}\right],
\end{aligned}
\end{equation}
where $\bar{\mc C}$ is the first part of the operator $\bar{\mc A}$ in \eqref{eq_expanded} and the second part is in $\bar{\mc D}$, including the linear constraints.
Lemma \ref{lemma_zero} guarantees that a zero of $\bar{\mc C}+\bar{\mc D}$ exists. Moreover, $\bar{\mc C}$ and $\bar{\mc D}$ in \eqref{eq_expanded2} have the following properties.

\begin{lemma}\label{lemma_op2}
Let Assumption \ref{ass_coco} hold and let $\Phi\succ 0$. The operators $\bar{\mc C}$ and $\bar{\mc D}$ in \eqref{eq_expanded2} have the following properties:
\begin{enumerate}
\item[(i)] $\bar{\mc C}$ is $\theta$-cocoercive where $0<\theta \leq\min \left\{\frac{1}{2 d^{*}}, \beta\right\}$ and $d^*$ is the maximum weighted degree of $\mc G^\lambda$;
\item[(ii)] The operator $\bar{\mc D}$ is maximally monotone;
\item[(iii)] $\Phi^{-1}\bar{\mc C}$ is $\theta\gamma$-cocoercive where $\gamma=\frac{1}{|\Phi^{-1}|}$;
\item[(iv)] $\Phi^{-1}\bar{\mc D}$ is maximally monotone.
\end{enumerate}
\end{lemma}
\begin{proof}
See Appendix \ref{app_op}.
\end{proof}


Also in this case, we use an approximation to compute the expected value, therefore, similarly to \eqref{eq_A_hat},
\begin{equation}\label{eq_A_hat2}
\hat{\mc C}:\left[\begin{array}{c}
(\bs{x},\xi) \\ 
\bs z\\
\bs \lambda
\end{array}\right] \mapsto\left[\begin{array}{c}
\hat F(\bs{x},\xi) \\ 
0\\
\bf L\bs\lambda+b
\end{array}\right].
\end{equation}
In this case, the SRFB algorithm is given by
\begin{equation}\label{algo2}
\begin{aligned} 
\bar{\bs \omega}^{k} &=(1-\delta)\bs \omega^k+\delta\bar{\bs \omega}^{k-1} \\ 
\bs \omega^{k+1} &=(\op{Id}+\Psi^{-1}\bar{\mc D})^{-1}(\bar{\bs \omega}^{k}-\Psi^{-1}\hat{\mc C}(\bs \omega^k))
\end{aligned}
\end{equation}
where the preconditioning matrix $\Psi$ is defined as
\begin{equation}\label{eq_phi2}
\Psi=\left[\begin{array}{ccc}
\alpha^{-1} & 0 & -{\bf{A}}^\top\\
0 & \nu^{-1} & -{\bf{L}}\\
-{\bf{A}} & -{\bf{L}} & \sigma^{-1}
\end{array}\right],
\end{equation} 
with $\alpha^{-1}$, $\nu^{-1}$, $\sigma^{-1}$ defined as in \eqref{eq_phi} and $\bf A$ and $\bf L$ are, respectively, the extended constraints and Laplacian matrix. We note that this type of preconditioning cannot be used for nonlinear coupling constraints $g(\bs x)\leq 0$ as in \eqref{collective_set}.

The distributed SRFB iterations read as in Algorithm \ref{algo_i2}. We note that Algorithm \ref{algo_i2} differs from Algorithm \ref{algo_i} in the computation of the dual variable $\bs\lambda^{k+1}$ which, in this case, depends also on the variables $\bs x^{k+1}$ and $\bs z^{k+1}$.

\subsection{Convergence analysis}
Since the matrix $\Psi$ must be positive definite \cite{belgioioso2018}, we postulate the following assumption.
\begin{assumption}[Step size sequence]\label{ass_step2}
The step size sequence is such that, given $\gamma>0$, for every agent $i\in\mc I$
$$\begin{aligned}
0&<\alpha_{i} \leq\textstyle{\left(\gamma+\max _{j\in\{1, \ldots, n_{i}\}}\sum\nolimits_{k=1}^{m}|[A_{i}^\top]_{j k}|\right)^{-1}} \\ 
0&<\nu_{i} \leq\textstyle{\left(\gamma+2 d_{i}\right)^{-1}}\\
0&<\sigma_{i} \leq\textstyle{\left(\gamma+2 d_{i}+\max _{j\in\{1, \ldots, m\}}\sum\nolimits_{k=1}^{n_{i}}|[A_{i}]_{j k}|\right)^{-1}}\\
\end{aligned}$$
where $[A_i^\top]_{jk}$ indicates the entry $(j,k)$ of the matrix $A_i^\top$.
Moreover,
$$\norm{\Psi^{-1}}\leq\tfrac{1}{\delta(2\ell_{\bar{\mc C}}-1)}$$
where $\delta$ is the averaging parameter and $\ell_{\bar{\mc C}}$ is the Lipschitz constant of $\bar{\mc C}$.
\fineass\end{assumption}

We are now ready to state the convergence result.
\begin{theorem}\label{theo_coco}
Let Assumptions \ref{ass_constr}--\ref{ass_delta}, \ref{ass_coco}, \ref{ass_step2} hold. Then the sequence $(\bs x_k)_{k\in\NN}$ generated by Algorithm \ref{algo_i2} with $\hat F$ as in \eqref{eq_F_SAA} converges a.s. to a v-SGNE of the game in \eqref{eq_game}.
\end{theorem}
\begin{proof}
See Appendix \ref{app_theo_coco}.
\end{proof}

\section{Stochastic Nash equilibrium problems}\label{sec_snep}

In this section we consider a non-generalized SNEP, namely, a SGNEP without shared constraints; see \cite{lei2020sync,lei2020} for recently proposed algorithms.
We consider that the local cost function of agent $i$ is defined as in \eqref{eq_cost_stoc} with $f_i(x_i)=\iota_{\Omega_i}(x_i)$. Assumptions \ref{ass_constr}--\ref{ass_sol} hold also in this case. 

The aim of each agent $i$, given the decision variables of the other agents $\bs{x}_{-i}$, is to choose a strategy $x_i$, that solves its local optimization problem, i.e.,
\begin{equation}\label{eq_game_SNE}
\forall i \in \mc I: \quad \min\limits _{x_i \in \Omega_i}  \JJ_i\left(x_i, \bs{x}_{-i}\right).
\end{equation}
As a solution, we aim to compute a stochastic Nash equilibrium (SNE), that is, a collective strategy $\bs x^*\in\bs{\Omega}$ such that for all $i \in \mc I$, 
$$\JJ_i(x_i^{*}, \boldsymbol x_{-i}^{*}) \leq \inf \{\JJ_i(y, \boldsymbol x_{-i}^{*})\; | \; y \in \Omega_i)\}.$$
We note that, compared to Definition \ref{def_GNE}, here we consider only local constraints.
Also in this case, we study the associated stochastic variational inequality (SVI) given by 
\begin{equation}\label{eq_svi_sne}
\langle \FF(\bs x^*),\bs x-\bs x^*\rangle\geq 0\text { for all } \bs x \in \bs{\Omega}
\end{equation}
where $\FF$ is the pseudogradient mapping as in \eqref{eq_grad}. 

The stochastic variational equilibria (v-SNE) of the game in (\ref{eq_game_SNE}) are defined as the solutions of the $\op{SVI}(\bs\Omega , \FF)$ in (\ref{eq_svi_sne}). 

\begin{remark}
A collective strategy $\bs x^*\in\bs{\mc X}$ is a Nash equilibrium of the game in \eqref{eq_game_SNE} if and only if $\bs x^*$ is a solution of the SVI in \eqref{eq_svi_sne} \cite[Proposition 1.4.2]{facchinei2007}, \cite[Lemma 3.3]{ravat2011}.
\fineass\end{remark}
The SRFB iterations for SNEPs are shown in Algorithm \ref{algo_i_sne}. 

\begin{algorithm}[t]
\caption{Stochastic Relaxed Forward Backward}\label{algo_i_sne}
Initialization: $x_i^0 \in \Omega_i$\\
Iteration $k$: Agent $i$ receives $x_j^k$ for all $j \in \mathcal{N}_{i}^{J}$, then updates:
$$\begin{aligned}
\bar{x}_i^{k} &=(1-\delta) x_i^k+\delta\bar{x}_i^{k-1} \\ 
x_i^{k+1}&=\op{proj}_{\Omega_i}[\bar x_i^k-\alpha_{i}\hat F_{i}(x_i^k, \boldsymbol{x}_{-i}^k,\xi_i^k)]
\end{aligned}$$
\end{algorithm}

\subsection{Convergence analysis}
In the SNEP case, we can consider sampling only one realization of the random variable, i.e, $S_k=1$:
\begin{equation}\label{eq_F_SA}
\begin{aligned}
\hat F(\bs x^k,\bs \xi^k)&=\op{col}(\hat F_i(\bs x^k,\bs \xi^k))\\
&=\op{col}(\nabla_{x_1}J_1(\bs x^k,\xi^k_1),\dots,\nabla_{x_N}J_N(\bs x^k,\xi^k_N)),
\end{aligned}
\end{equation}
where $\bs\xi^{k} =\op{col}(\xi_1^{k},\dots,\xi_{N}^k)\in\RR^N$ is a collection of i.i.d. random variables drawn from $\PP$.
Taking fewer samples is less computationally expensive but we have to make some further assumptions on the pseudogradient mapping.

\begin{assumption}[Cut property]\label{ass_cut}
$\FF$ in \eqref{eq_grad} is such that:
\begin{equation}\label{cut}
\left.\begin{array}{r}
\bs x^{*} \in \mathrm{SOL}(\FF, \bs\Omega) \\ 
\bar {\bs x} \in C \\ 
\bar {\bs x} \neq \bs x^{*} \\ 
\left\langle \FF(\bar {\bs x}), \bar {\bs x}-\bs x^{*} \right\rangle= 0
\end{array}\right\} 
\Longrightarrow \bar{\bs x} \in \mathrm{SOL}(\FF,\bs\Omega)\vspace{-.5cm}
\end{equation}
\fineass\end{assumption}
\begin{remark}\label{remark_cut}
The cut property means that, given a solution $\bs x^*$, it can be verified if another point $\bar{\bs x}$ is also a solution by looking only at $\bar{\bs x}$ and $\bs x^*$ instead of comparing $\bar{\bs x}$ with all the points in $\bs{\mc X}$. A very intuitive example is the search for a minimum of a single-valued function \cite{iusem1998}. 

The class of mappings that satisfy this assumption is that of paramonotone (or monotone$^+$) operators. A paramonotone operator is a monotone operator such that for all $\bs x,\bs y\in\bs{\mc X}$
$$\langle \FF(\bs x)-\FF(\bs y),\bs x-\bs y\rangle=0 \Rightarrow \FF(\bs x)=\FF(\bs y).$$
This property does not hold in general for monotone operators. It holds for strongly and strictly monotone operators, because in this cases there is only one solution \cite[Theorem 2.3.3]{facchinei2007}, and for cocoercive operators. In fact, strict monotonicity implies paramonotonicity that in turn implies monotonicity \cite[Definition 2.1]{iusem1998}. The same holds for cocoercive operators that are also paramonotone and consequently monotone \cite[Definition 2.3.9]{facchinei2007}.
We refer to \cite{iusem1998, crouzeix2000} for a deeper insight on this class of operators.
\fineass\end{remark}

\begin{assumption}[Bounded pseudogradient]\label{ass_bounded}
$\FF$ is bounded, i.e., there exists $B>0$ such that for all $\bs x\in \bs{\mc X}$ $\normsq{\FF(\bs x)}\leq B_\FF.$
\end{assumption}
Even if this assumption is quite strong, it is reasonable in our game theoretic framework. On the other hand, we do not require $\FF$ to be Lipschitz continuous, which is practical since computing the Lipschitz constant is difficult in general. 

With a little abuse of notation, we denote the approximation error again with $\epsilon_k=\hat F(\bs x^k,\xi_k)-\FF(\bs x^k).$

Concerning the assumptions on the stochastic error, we still suppose that it has zero expected value (Assumption \ref{ass_error}) but we do not need an explicit bound on the variance. 
\begin{assumption}[Parameter and step sizes]\label{ass_step}
The averaging parameter is such that $\delta\in(0,1)$.
The step size is square summable and such that
$$\sum_{k=0}^\infty\gamma_k=\infty, \;\sum_{k=0}^\infty\gamma_k^2<\infty \text{ and }\sum_{k=0}^\infty\gamma_k^2\,\EEk{\normsq{\epsilon_k}}<\infty.\vspace{-.4cm}$$
\fineass\end{assumption}
It follows from Assumption \ref{ass_step} that we can take a larger bound on the averaging parameter $\delta$. Since $\delta$ is no longer related to the golden ratio, the algorithm reduces to a relaxed FB iteration. Moreover, we note that in this case we must take a vanishing step size sequence to control the stochastic error.

We now state the main convergence result of this section.

\begin{theorem}\label{theo_sne}
Let Assumptions \ref{ass_constr}--\ref{ass_Jraph}, \ref{ass_error}, \ref{ass_weaker}, \ref{ass_cut}--\ref{ass_step} hold. Then, the sequence $(\bs x_k)_{k\in\NN}$ generated by Algorithm \ref{algo_i_sne} with $\hat F$ as in \eqref{eq_F_SA} converges to a solution of the game in \eqref{eq_game_SNE}.
\end{theorem}
\begin{proof} 
See Appendix \ref{app_theo_sne}.
\end{proof}

\begin{remark}
We note that Theorem \ref{theo_sne} holds also in the deterministic case, under the same assumptions with the exception of those on the stochastic error (that is not present). 

Formally, under Assumptions \ref{ass_constr}--\ref{ass_Jraph}, \ref{ass_cut}--\ref{ass_step}, Algorithm \ref{algo_i_sne} converges to a v-NE of the game in \eqref{eq_game_SNE}.
Equivalently, one can use \cite[Algorithm 1]{malitsky2019} to find a deterministic NE.
\fineass\end{remark}

\subsection{Discussion on further monotonicity assumptions}\label{sec_cor}

In this section we discuss some consequences of Theorems \ref{theo_sgne} and \ref{theo_sne}. In particular, we discuss different monotonicity notions that can be used to find a SNE in relation with the two possible approximation schemes.

First of all, Algorithm \ref{algo_i} with the approximation as in \eqref{eq_F_SAA} can be used also for SNEPs.

\begin{corollary}\label{cor_vr_snep}
If Assumptions \ref{ass_constr}--\ref{ass_sol}, \ref{ass_mono}--\ref{ass_step1} hold. Then, the sequence $(\bs x^k)_{k\in\NN}$ generated by Algorithm \ref{algo_i} with $\hat F$ as in \eqref{eq_F_SAA} converges to a v-SNE of the game in \eqref{eq_game_SNE}.
\end{corollary}
\begin{proof}
Set $g\equiv0$ and apply Theorem \ref{theo_sgne}.
\end{proof}
\begin{remark}
In Corollary \ref{cor_vr_snep} as well, the condition presented in Remark \ref{remark_weak} can be used instead of monotonicity.
\fineass\end{remark}

The same result holds in the case of cocoercive mappings but in this case Assumption \ref{ass_coco} can be reduced to the cut property for the pseudogradient (Assumption \ref{ass_cut}).
\begin{corollary}
Let Assumptions \ref{ass_constr}--\ref{ass_sol}, \ref{ass_batch}--\ref{ass_delta}, \ref{ass_coco}, \ref{ass_step2} hold. Then, the sequence $(\bs x_k)_{k\in\NN}$ generated by Algorithm \ref{algo_i2} with $\hat F$ as in \eqref{eq_F_SAA} converges a.s. to a v-SNE for the game in \eqref{eq_game_SNE}.
\end{corollary}
\begin{proof}
Set $A=0$ and $b=0$ and apply Theorem \ref{theo_coco}.
\end{proof}

Besides the cut property, there are other assumptions that can be considered.

\begin{assumption}[Weak sharpness]\label{ass_sharp}
$\FF$ satisfies the weak sharpness property, i.e. for all $\bs x\in\bs{\mc X}$, $\bs x^*\in SOL(\FF,\bs\Omega)$ and for some $c>0$
$$\langle \FF(\bs x^*),\bs x-\bs x^*\rangle\geq c\min_{\bs x^*\in SOL(\FF,\Omega)}\norm{\bs x-\bs x^*}$$
\end{assumption}

\begin{remark}
Assumption \ref{ass_sharp} is stronger than that in Remark \ref{remark_weak} and it is often used in addition to monotonicity \cite{kannan2019,cui2019}. It is sometimes considered a property of the solution set and it is implied by paramonotonicity \cite[Theorem 4.1]{marcotte1998}.
\fineass\end{remark}

\begin{corollary}\label{cor_weak}
Let Assumptions \ref{ass_constr}--\ref{ass_sol}, \ref{ass_error}, \ref{ass_step}, \ref{ass_sharp} hold. Then, the sequence $(\bs x_k)_{k\in\NN}$ generated by Algorithm \ref{algo_i_sne} with $\hat F$ as in \eqref{eq_F_SA} converges to a v-SNE of the game in \eqref{eq_game_SNE}.
\end{corollary}
\begin{proof}
See Appendix \ref{app_sec_cor}.
\end{proof}
\begin{remark}\label{remark_acute}
As a technical assumption in addition to monotonicity, we can also consider the acute angle relation, i.e., $\langle \FF(\bs x),\bs x-\bs x^*\rangle>0$ for all $\bs x\in\bs{\mc X}$ and $\bs x^*\in SOL(\FF,\bs\Omega)$, $\bs x\neq\bs x^*$, also known as variational stability \cite{mertikopoulos2019}. It is implied by strict pseudomonotonicity which in turn is implied by strict monotonicity \cite[Definition 2]{kannan2019}, \cite[Corollary 2.4]{mertikopoulos2019}. It is stronger than the assumption in Remark \ref{remark_weak} since the condition is satisfied with the strict inequality.\\
Formally, if Assumptions \ref{ass_constr}--\ref{ass_sol}, \ref{ass_error}, \ref{ass_step} and the acute angle relation hold, then, the sequence $(\bs x_k)_{k\in\NN}$ generated by Algorithm \ref{algo_i_sne} with $\hat F$ as in \eqref{eq_F_SA} converges a.s. to a v-SNE of the game in \eqref{eq_game_SNE}.
We propose a proof in Appendix \ref{app_sec_cor}. The same result for $\delta=0$ is established in \cite[Theorem 4.7]{mertikopoulos2019}.
\end{remark}

%
%
%

\section{Numerical simulations}\label{sec_sim}
Let us now propose some numerical simulations to corroborate the theoretical analysis. We compare our algorithm with the stochastic distributed preconditioned forward--backward (SpFB) \cite{franci2019}, forward--backward--forward (SFBF) \cite{franci2019fbf,staudigl2019}, extragradient (SEG) \cite{iusem2017} and projected reflected gradient (SPRG) \cite{cui2019,cui2016} algorithms, using variance reduction.

We present two sets of simulations: a Cournot game and an academic example. While the first is a realistic application to an electricity market with market capacity constraints, the second is built to show the advantages of the SRFB algorithm.

All the simulations are performed on Matlab R2019a with a 2,3 GHz Intel Core i5 and 8 GB LPDDR3 RAM.
%
\subsection{Illustrative example}
We start with the built up example, that is, a monotone (non-cocoercive) stochastic Nash equilibrium problem with two players with strategies $x_1$ and $x_2$ respectively, and pseudogradient mapping $\FF(x_1,x_2)=(R_1(\xi)x_2,-R_2(\xi)x_1)^\top$.
The mapping is monotone and the random variables are sampled from a normal distribution with mean 1 and finite variance. The problem is unconstrained and the optimal solution is $(0,0)$. The step sizes are taken to be the highest possible. 
As one can see from Fig. \ref{plot_sol_complex}, the SpFB does not converge in this case because stronger monotonicity properties on the mapping should be taken. Moreover, we note that the SPRG is not guaranteed to converge under mere monotonicity. From Fig. \ref{plot_sec_complex} instead, we note that the SRFB algorithm is less computationally expensive than the EG.

\begin{figure}[t]
\centering
\includegraphics[width=.45\textwidth]{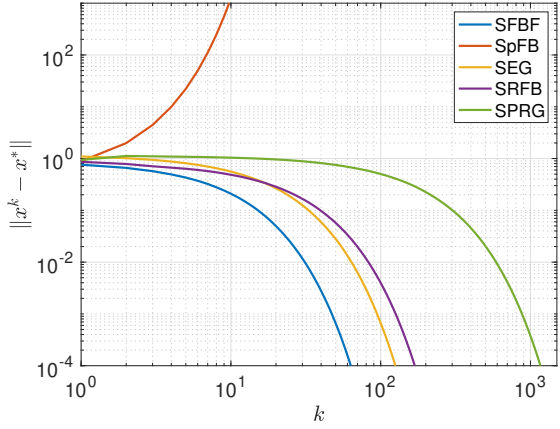}
\caption{Distance of the primal variable from the solution.}\label{plot_sol_complex}
\end{figure}

\begin{figure}[t]
\centering
\includegraphics[width=.45\textwidth]{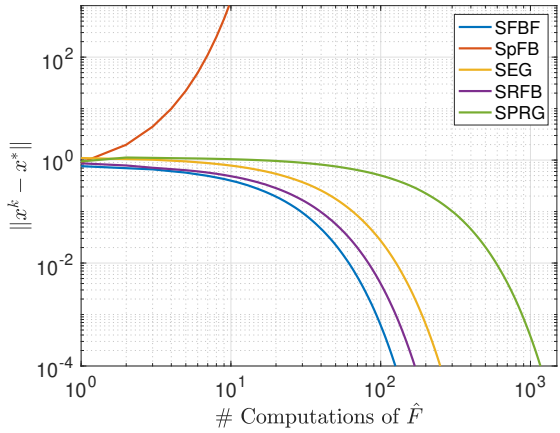}
\caption{Distance of the primal variable from the solution versus number of evaluations of $\hat F$ in \eqref{eq_F_SAA}}.\label{plot_sec_complex}
\end{figure}
\subsection{Case study: Network Cournot game}
We consider the network Cournot game proposed in \cite{malitsky2019} with the addition of markets capacity constraints \cite{yi2019,yu2017} which may model the electricity market, the gas market or the transportation system \cite{watling2006,henrion2007}.
Let us consider a set of $N$ companies that sell a commodity in a set of $m$ markets. Each company decides the quantity $x_i$ of product to be delivered in the $n_i$ markets it is connected with. Each company has a local cost function $c_i(x_i)$ related to the production of the commodity. We assume that the cost function is not uncertain as the companies should know their own cost of production. \\
Since the markets have a bounded capacity $b=[b_1,\dots,b_m]$, the collective constraints are given by $A\boldsymbol x\leq b$ where $A=[A_1,\dots,A_N]$. Each $A_i$ indicates in which markets each company participates. The prices are collected in a mapping $P:\RR^m\times \Xi\to\RR$ that denotes the inverse demand curve. 
The random variable $\xi\in\Xi$ represents the demand uncertainty. The cost function of each agent is therefore given by
\begin{equation}\label{eq_sim}
\textstyle{\JJ_i(x_i,x_{-i},\xi)=c_i(x_i)-\EE\left[P(\bs x,\xi)\sum_{i\in\mc I}x_i\right].}
\end{equation}


As a numerical setting, we consider a set of 20 companies and 7 markets, connected as in \cite[Fig. 1]{yi2019}. Following \cite{yi2019}, we suppose that the dual variables graph is a cycle graph with the addition of the edges $(2,15)$ and $(6,13)$. 
Each company $i$ has local constraints of the form $0 < x_i \leq \theta_i$ where each component of $\theta_i$ is randomly drawn from $[1, 1.5]$. The maximal capacity $b_j$ of a market $j$ is randomly drawn from $[0.5, 1]$. The local cost function of company $i$ is 
$$\textstyle{c_i(x_i) =q_i^\top x_i+\frac{\beta_i}{\beta_i+1} \pi_{i}^{\frac{1}{\beta_i}} \sum_{j=1}^{n_i} ([x_i]_j)^{\frac{\beta_i+1}{\beta_i}}}$$
where $[x_i]_j$ indicates the $j$ component of $x_i$.
$\pi_i$ is randomly drawn from $[0.5,5]$, and each component of $q_i$ is randomly drawn from $[1,100]$. 
Similarly to \cite{malitsky2019}, we assume that the inverse demand function is of the form
$$\textstyle{P(\bs x,\xi)=\Lambda(\xi)^{\frac{1}{\gamma}}\left(\sum_{i\in\mc I}x_i\right)^{-\frac{1}{\gamma}}}$$
where $\gamma=1.1$ and $\Lambda(\xi)$ is drawn following a normal distribution with mean 5000 and finite variance.
We note that the mapping in \eqref{eq_sim} is monotone but it may be not Lipschitz continuous depending on $\beta$ and $\gamma$.
\begin{figure}[t]
\centering
\includegraphics[width=.45\textwidth]{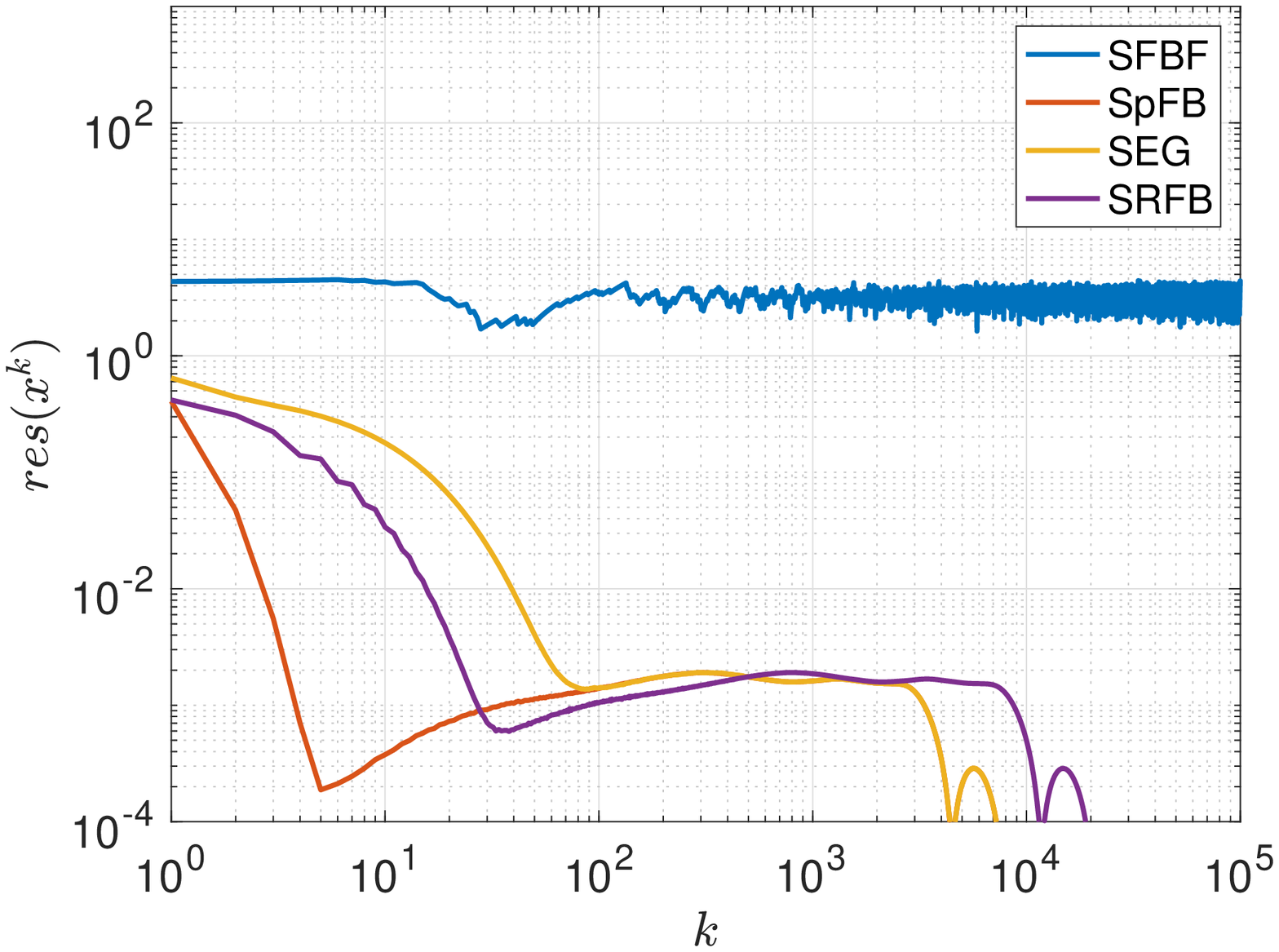}
\caption{Residual distance of the primal variable from the solution.}\label{plot_sol}
\end{figure}

\begin{figure}[t]
\centering
\includegraphics[width=.45\textwidth]{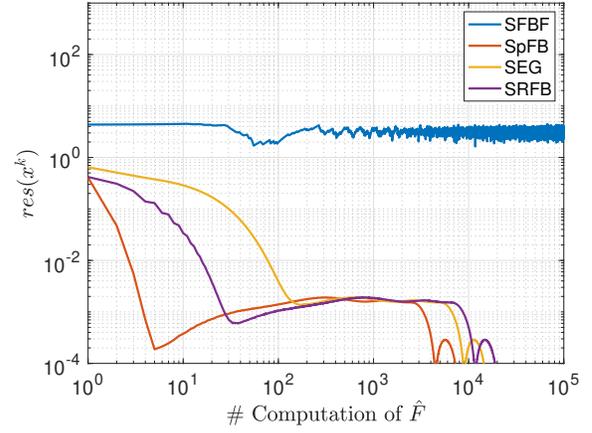}
\caption{Residual distance of the primal variable from the solution versus number of evaluations of $\hat F$ in \eqref{eq_F_SAA}.}\label{plot_time} 
\end{figure}

We simulate the SpFB, SFBF, SEG and SRFB to make a comparison using the SA scheme with variance reduction. Since the mapping is not Lipschitz continuous, we tune the step sizes to be half of the minimum step that causes instability.
The plots in Fig. \ref{plot_sol} and \ref{plot_time} show respectively the residual of $\bs x_k$ ($\op{res}(\bs x_k)$) that measure the distance from $\bs x_k$ being a solution, and the number of computations of the approximations $\hat F$ in \eqref{eq_F_SAA} of the pseudogradient needed to reach a solution. As one can see, our algorithm is slower than the SpFB as ours involves the averaging step but it is faster than the EG scheme. Remarkably, the fact that the mapping is only monotone and not Lipschitz continuous prevent the SFBF from converging but it does not affect the other algorithms.

\section{Conclusion}

The stochastic relaxed forward--backward algorithm is applicable to stochastic (generalized) Nash equilibrium seeking in merely monotone games. To approximate the expected valued pseudogradient, the stochastic approximation scheme (with or without variance reduction) can be used to guarantee almost sure convergence to an equilibrium.

Our stochastic relaxed forward--backward algorithm is the first distributed algorithm with single proximal computation and single approximated pseudogradient computation per iteration for merely monotone stochastic games.

It remains an open question whether for monotone SGNEPs the stochastic approximation with only one random sample per iteration can guarantee almost sure convergence to an equilibrium, instead of the variance reduced approach.
We also leave for future work a comprehensive comparison between the SRFB algorithm and the most popular fixed-step algorithms for SVIs and SGNEPs (especially, SEG and SFBF) in terms of computational complexity and convergence speed.

\section*{Acknowledgments}
The authors thank Alfiya Kulmukhanova for preliminary discussions on relaxed forward--backward algorithms.

\appendices
\section{Properties of the extended operators}\label{app_op}
\begin{proof}[Proof of Lemma \ref{lemma_zero}]
The proof of (i) can be obtained similarly to \cite[Theorem 2]{yi2019}. 
Concerning (ii), given Assumption \ref{ass_constr}--\ref{ass_sol}, the game in \eqref{eq_game} has at least one solution $\bs x^*$, therefore, there exists a $\bs \lambda^*\in\RR^m_{\geq0}$ such that the KKT conditions in \eqref{eq_KKT_VI} are satisfied \cite[Theorem 3.1]{auslender2000}. It follows that $\op{zer}(\mc A+\mc B)\neq\emptyset$. The existance of $\bs z^*$ such that $\op{col}(\bs x^*,\bs z^*,\bs\lambda^*)\in\op{zer}(\bar{\mc A}+\bar{\mc B})$ follows using some properties of the normal cone and of the Laplacian matrix as a consequence of Assumption \ref{ass_Jraph} \cite[Theorem 2]{yi2019}.
\end{proof}

\begin{proof}[Proof of Lemma \ref{lemma_op}]
$\bar{\mc A}=\mc A_1+\mc A_2$ is given by a sum, therefore it is monotone if both the addend are \cite[Proposition 20.10]{bau2011}. $\mc A_2$ is monotone because of \cite[Theorem 1]{rockafellar1970} and monotonicity of $\mc A_1$ follows from 
$$\begin{aligned}
\langle\mc A_1(\bs\omega_1)-\mc A_1(\bs\omega_2),\bs\omega_1-\bs\omega_2\rangle&=\langle\FF(\bs x_1)-\FF(\bs x_2),\bs x_1-\bs x_2\rangle\\
&+\langle \bf L\bs\lambda_1-\bf L\bs \lambda_2,\bs \lambda_1-\bs \lambda_2\rangle\geq 0,
\end{aligned}$$
since Assumption \ref{ass_mono} holds and $\bf L$ is cocoercive by the Baillon-Haddard Theorem and, therefore, monotone \cite[Example 20.5]{bau2011}. To show that $\bar{\mc A}$ is Lipschitz continuous, we use the fact that $\FF$ is $\ell_F$-Lipschitz and $\bf L$ is $\ell_L$-Lipschitz continuous:
$$\begin{aligned}
\norm{\bar{\mc A}_1(\bs\omega_1)-\bar{\mc A}_1(\bs\omega_2)}&\leq\norm{\FF(\bs x_1)-\FF(\bs x_2)}+\norm{\bf L\bs\lambda_1-\bf L\bs \lambda_2}\\
&\leq(\ell_F+\ell_L)(\norm{\bs x_1-\bs x_2}+\norm{\bs \lambda_1-\bs \lambda_2}).
\end{aligned}$$
Similarly we can prove that the skew symmetric part is $\ell_{\mc A_2}=\ell_L+\ell_g+B_{\nabla g}$-Lipschitz continuous (with constant that depends on $A$ and $L$) from which  it follows that $\bar{\mc A}$ is $\ell_{\bar{\mc A}}=\ell_{\mc A_1}+\ell_{\mc A_2}$-Lipschitz continuous.\\
$\bar{\mc B}$ is maximally monotone by \cite[Proposition 20.23]{bau2011} because $\partial f$ is maximally monotone by Assumption \ref{ass_J} and Moreau Theorem \cite[Theorem 20.25]{bau2011} and the Normal cone is maximally monotone \cite[Example 20.26]{bau2011}.\\
The fact that $\Phi^{-1}\bar{\mc A}$ is monotone follows from the fact that $\bar{\mc A}$ is monotone:
$\langle\Phi^{-1}(\bar{\mc A}(\bs\omega_1)-\Phi^{-1}\bar{\mc A}(\bs\omega_2),\bs\omega_1-\bs\omega_2\rangle_\Phi=\langle\bar{\mc A}(\bs\omega_1)-\bar{\mc A}(\bs\omega_2),\bs\omega_1-\bs\omega_2\rangle\geq0.$
Similarly it holds that $\Phi^{-1}\bar{\mc A}$ is Lipschitz continuous and that $\Phi^{-1}\bar{\mc B}$ is maximally monotone.
\end{proof}

\begin{proof}[Proof of Lemma \ref{lemma_op2}]
First we notice that $\norm{L}\geq 2d^*$ and that by the Baillon--Haddard theorem the Laplacian operator is $\tfrac{1}{2d^*}$-cocoercive. Then Statement 1) follow by this computation:
$$\begin{aligned}
\langle\bar{\mc C}&(\bs\omega_1)-\bar{\mc C}(\bs\omega_2),\bs\omega_1-\bs\omega_2\rangle\\
&=\langle \FF(\bs x_1)-\FF(\bs x_2),x_1-x_2\rangle+\langle {\bf{L}}\bs \lambda_1-{\bf{L}}\bs \lambda_2,\bs \lambda_1-\bs \lambda_2\rangle\\
&\geq\beta\normsq{\FF(\bs x_1)-\FF(\bs x_2)}+\tfrac{1}{2d^*}\normsq{{\bf{L}}\bs \lambda_1-{\bf{L}}\bs \lambda_2}\\
&\geq\min\left\{\beta,\tfrac{1}{2d^*}\right\}\left(\normsq{\FF(\bs x_1)-\FF(\bs x_2)}+\normsq{{\bf{L}}\bs \lambda_1-{\bf{L}}\bs \lambda_2}\right)\\
&\geq\theta\normsq{\bar{\mc C}(\bs\omega_1)-\bar{\mc C}(\bs\omega_2)}.
\end{aligned}$$
The operator $\bar{\mc D}$ is given by a sum, therefore it is maximally monotone if both the addend are \cite[Proposition 20.23]{bau2011}. The first part is maximally monotone because the normal cone is and the second part is a skew symmetric matrix \cite[Corollary 20.28]{bau2011}.
Statement 3) follows from Statement 1) and 4) follows from 2) \cite[Lemma 7]{yi2019}.
\end{proof}

\section{Useful lemmas}
We here recall some known facts about norms and sequence of random variables. Moreover, we include two preliminary results that are useful for the forthcoming convergence proofs.
\paragraph{Norm properties}
Now we recall some property of the norms that we will use in the proofs. 
We use the cosine rule (or Pythagorean identity)
\begin{equation}\label{cosine}
\langle x,y\rangle=\tfrac{1}{2}\left(\langle x,x\rangle+\langle y,y\rangle-\norm{x-y}^2\right)
\end{equation}
and the following two property of the norm \cite[Corollary 2.15]{bau2011}, $\forall a, b \in \mathcal{X}$, $\forall \alpha \in \mathbb{R}$
\begin{equation}\label{norm_conv}
\|\alpha a+(1-\alpha) b\|^{2}=\alpha\|a\|^{2}+(1-\alpha)\|b\|^{2}-\alpha(1-\alpha)\|a-b\|^{2},
\end{equation}
\begin{equation}\label{norm_sum}
\|a+b\|^{2}\leq2\|a\|^{2}+2\|b\|^{2}.
\end{equation}

\paragraph{Property of the projection and proximal operator}
By \cite[Proposition 12.26]{bau2011}, the projection operator and the proximity operators satisfy, respectively, the following inequalities. Let $C$ be a nonempty closed convex set and let $g$ be a proper lower semicontinuous function, then, for all $x,y\in C$
\begin{equation}\label{proj}
\bar x=\op{proj}_C(x)\Leftrightarrow\langle \bar x-x, y-\bar x\rangle \geq 0 .
\end{equation}
\begin{equation}\label{prox}
\bar x=\op{prox}_f(x)\Leftrightarrow\langle \bar x-x, y-\bar x\rangle \geq f(\bar x)-f(y) 
\end{equation}
Moreover, by \cite[Proposition 16.44]{bau2011}, it holds that
\begin{equation}\label{resolvent}
\op{prox}_f=(\op{Id}+\partial f)^{-1}.
\end{equation}
\paragraph{Sequence of random variables}
We now recall some results concerning sequences of random variables, given the probability space $(\Xi, \mc F, \PP)$.
The Robbins-Siegmund Lemma is widely used in literature to prove a.s. convergence of sequences of random variables. It first appeared in \cite{RS1971}. 
\begin{lemma}[Robbins-Siegmund Lemma, \cite{RS1971}]\label{lemma_RS}
Let $\mc F=(\mc F_k)_{k\in\NN}$ be a filtration. Let $\{\alpha_k\}_{k\in\NN}$, $\{\theta_k\}_{k\in\NN}$, $\{\eta_k\}_{k\in\NN}$ and $\{\chi_k\}_{k\in\NN}$ be nonnegative sequences such that $\sum_k\eta_k<\infty$, $\sum_k\chi_k<\infty$ and let
$$\forall k\in\NN, \quad \EE[\alpha_{k+1}|\mc F_k]+\theta_k\leq (1+\chi_k)\alpha_k+\eta_k \quad a.s.$$
Then $\sum_k \theta_k<\infty$ and $\{\alpha_k\}_{k\in\NN}$ converges a.s. to a nonnegative random variable.
\end{lemma}

We also need this result for $L_p$ norms, known as Burkholder-Davis-Gundy inequality \cite{stroock2010}.
\begin{lemma}[Burkholder--Davis--Gundy inequality]\label{BDG_Lemma}
Let $\{\mc F_k\}$ be a filtration and $\{U_k\}_{k\in\NN}$ a vector-valued martingale relative to this filtration. Then, for all $p\in [1,\infty)$, there exists a universal constant $c_p > 0$ such that for every $k\geq1$
$$\textstyle{\EE\left[\left(\sup _{0 \leq i \leq N}\left\|U_{i}\right\|\right)^{p}\right]^{\frac{1}{p}} \leq c_{p} \EE\left[\left(\sum_{i=1}^{N}\left\|U_{i}-U_{i-1}\right\|^{2}\right)^{\frac{p}{2}}\right]^{\frac{1}{p}}.}$$
\end{lemma}

When combined with Minkowski inequality, we obtain for all $p\geq2$ a constant $C_p >0 $ such that for every $k\geq1$
$$\textstyle{\EE\left[\left(\sup _{0 \leq i \leq N}\norm{U_{i}}\right)^{p}\right]^{\frac{1}{p}} \leq C_p\sqrt{\sum_{k=1}^{N} \EE\left(\norm{U_{i}-U_{i-1}}^{p}\right)^{\frac{2}{p}}}.}$$

\paragraph{Preliminary results}
Given Lemma \ref{BDG_Lemma}, we prove a preliminary result on the variance of the stochastic error.
\begin{proposition}\label{prop_variance}
For all $k\in\NN$, if Assumption \ref{ass_variance} holds, we have
$$\EE\left[\norm{\epsilon_k}^2|\mc F_k\right]\leq\frac{C(\sigma(\bs x^*)^2+\sigma_0^2\normsq{\bs x-\bs x^*})}{S_k} \text{ a.s.}.$$
\end{proposition}
\begin{proof}
We first prove that, for $p\geq 2$ and $1\leq q\leq p$
$$\EE\left[\norm{\epsilon_k}^q|\mc F_k\right]^{\frac{1}{q}}\leq\frac{c_q(\sigma(\bs x^*)+\sigma_0\norm{\bs x-\bs x^*})}{\sqrt{S_k}}.$$
Let us define the process $\{M_i^S(\bs x)\}_{i=0}^S$ as $M_0(x)=0$ and for $1\leq i\leq S$
$$\textstyle{M_i^S(x)=\frac{1}{S}\sum_{j=1}^iF_k(\bs x,\xi_j)-\FF(\bs x)}.$$
Let $\mc F_i=\sigma(\xi_1,\dots,\xi_i)$. Then $\{M_i^S(\bs x),\mc F_i\}_{i=1}^S$ is a martingale starting at $0$.
Let
$$\begin{aligned}
\Delta M_{i-1}^S(\bs x)&=M_i^S(\bs x)-M_{i-1}^S(\bs x)\\
&=F_k(\bs x,\xi_i)-\FF(\bs x).
\end{aligned}$$
Then, by Equation (\ref{eq_variance}), we have
\begin{equation*}
\begin{aligned}
\EEx{\norm{\Delta M_{i-1}^S(\bs  x)}^q}^{\frac{1}{q}}&=\frac{1}{S}\EEx{\norm{F_k(\bs x,\xi_i)-\FF(x)}^q}^{\frac{1}{q}}\\
&\leq \frac{\sigma(\bs x)+\sigma_0\norm{\bs x-\bs x^*}}{S}.
\end{aligned}
\end{equation*}
By applying Lemma \ref{BDG_Lemma}, we have
\begin{equation*}
\begin{aligned}
\EEx{\norm{M_S^S(\bs x)}^q}^{\frac{1}{q}}&\leq \textstyle{c_q \sqrt{\sum_{i=1}^{N} \EEx{\norm{\frac{F_k(\bs x, \xi_i)-\FF(x)}{S}}^q}^{\frac{2}{q}}}}\\
&\leq\textstyle{ c_q\sqrt{\frac{1}{S^2} \sum_{i=1}^{N} \EEx{\norm{F_k(\bs x, \xi_i)-\FF(x)}^q}^{\frac{1}{q}}}}\\
&\leq \frac{c_q(\sigma(x)+\sigma_0\norm{x-x^*})}{\sqrt{S}}.
\end{aligned}
\end{equation*}
We note that $M_{S_k}^{S_k}(x^k)=\epsilon_k$, hence 
\begin{equation*}
\EEk{\norm{\epsilon_k}^q}^{\frac{1}{q}}\leq\frac{c_q(\sigma(\bs x)+\sigma_0\norm{\bs x-\bs x^*}}{\sqrt{S_k}}. 
\end{equation*}
Then, the claim follows by noting that $\EEk{\norm{\epsilon}^{2q}}^{\frac{1}{q}}=\EEk{\norm{\epsilon}^{q}}^{\frac{2}{q}}$ and $C=2c_q^2$.
\end{proof}
\begin{remark}
If Proposition \ref{prop_variance} holds, then it follows that
$$\EE\left[\norm{\Phi^{-1}\varepsilon_k}_\Phi^2|\mc F_k\right]\leq\frac{C\norm{\Phi^{-1}}(\sigma(\bs x^*)^2+\sigma_0^2\normsq{\bs x-\bs x^*})}{S_k}.\vspace{-.4cm}$$
\fineass\end{remark}

In the next Lemma, we collect some inequalities that follow from the definition of the algorithm in \eqref{algo}.

\begin{lemma}\label{lemma_algo}
Let $(\bs\omega_k,\bar{\bs\omega}_k)_{k\in\NN}$ be generated by Algorithm \ref{algo_i} defined as in \eqref{algo}. Then, the following equations hold:
\begin{itemize}
\item[(i)] $\bs \omega^k-\bar {\bs \omega}^{k-1}=\tfrac{1}{\delta}(\bs \omega^k-\bar {\bs \omega}^k)$;
\item[(ii)] $\bs\omega^{k+1}-\bs\omega^*=\tfrac{1}{1-\delta}(\bar{\bs\omega}^{k+1}-\bs\omega^*)-\tfrac{\delta}{1-\delta}(\bar{\bs\omega}^k-\bs\omega^*)$;
\item[(iii)] $\frac{\delta}{(1-\delta)^2}\norm{\bar{\bs\omega}^{k+1}-{\bs\omega}^k}^2=\delta\norm{\bs\omega^{k+1}-{\bs\omega}^k}^2$.
\end{itemize}
\end{lemma}
\begin{proof}
It follows immediately from \eqref{algo}.
\end{proof}

\section{Proofs of Section \ref{sec_conv_vr}}\label{app_sec_algo}
The proof uses the $\Phi$-induced norm and inner product and finds its inspiration in \cite{malitsky2019, staudigl2019, iusem2017}.
\begin{proof}[Proof of Theorem \ref{theo_sgne}]
First, let us define 
$\textstyle{H(\bs x,\bs\lambda) = \sum_{i\in\mc I}f_i(x_i)+\iota_{\RR_{\geq0}}(\lambda_i)}$
and note that
$\textstyle{\Phi^{-1}\bar{\mc B}=\alpha\partial\left[\sum_{i\in\mc I}f_i(x_i)+\iota_{\RR_{\geq0}}(\lambda_i)\right]=\alpha\partial H.}$
By \eqref{resolvent}, $(\op{Id}+\Phi^{-1}\bar{\mc B})^{-1}=(\op{Id}+\alpha\partial H)^{-1}=\op{prox}_{\alpha H}$. Therefore, by using the property of proximal operators in (\ref{prox}), we have that
\begin{equation}\label{stepNk1}
\begin{aligned}
\langle \bs\omega^{k+1}-\bar {\bs \omega}^k &+\Phi^{-1}\hat{\mc A}(\bs \omega^k,\xi^k),\bs\omega^*-\bs\omega^{k+1}\rangle\geq\\
&\geq \alpha (H(\bs \omega^{k+1})-H(\bs\omega^*))
\end{aligned}
\end{equation}
\begin{equation}\label{stepNk2}
\begin{aligned}
\langle \bs \omega^k-\bar {\bs \omega}^{k-1} &+\Phi^{-1}\hat{\mc A}(\bs\omega^{k-1},\xi^{k-1}),\bs\omega^{k+1}-\bs \omega^k\rangle\geq\\
&\geq\alpha(H(\bs \omega^{k})-H(\bs\omega^{k+1})).
\end{aligned}
\end{equation}
By using Lemma \ref{lemma_algo}(i), (\ref{stepNk2}) becomes
\begin{equation}\label{stepNk3}
\begin{aligned}
\langle \tfrac{1}{\delta}(\bs \omega^k-\bar {\bs \omega}^k)&+\Phi^{-1}\hat{\mc A}(\bs\omega^{k-1},\xi^{k-1}),\bs\omega^{k+1}-\bs \omega^k\rangle\geq \\
&\geq\alpha(H(\bs \omega^{k})-H(\bs\omega^{k+1})).
\end{aligned}
\end{equation}
Then, by adding (\ref{stepNk1}) and (\ref{stepNk3}) we obtain
\begin{equation}\label{stepNk4}
\begin{aligned}
&\langle \bs\omega^{k+1}-\bar {\bs \omega}^k +\Phi^{-1}\hat{\mc A}(\bs \omega^k,\xi^k),\bs\omega^*-\bs\omega^{k+1}\rangle+\\
&+\langle \tfrac{1}{\delta}(\bs \omega^k-\bar {\bs \omega}^k)+\Phi^{-1}\hat{\mc A}(\bs\omega^{k-1},\xi^{k-1}),\bs\omega^{k+1}-\bs \omega^k\rangle\geq\\
&\geq \alpha(H(\bs \omega^{k})-H(\bs\omega^*)).
\end{aligned}
\end{equation}
Now, we use the cosine rule in (\ref{cosine}):
\begin{equation*}\label{cos_zeta}
\begin{aligned}
&\langle \bs\omega^{k+1}-\bar {\bs \omega}^k,\bs\omega^*-\bs\omega^{k+1}\rangle=\\
&-\tfrac{1}{2}\left(\norm{\bs\omega^{k+1}-\bar {\bs \omega}^k}^2+\norm{\bs\omega^{k+1}-\bs\omega^*}^2-\norm{\bs\omega^*-\bar {\bs \omega}^k}^2\right)\\
&\langle \tfrac{1}{\delta}(\bs \omega^k-\bar {\bs \omega}^k),\bs\omega^{k+1}-\bs \omega^k\rangle=\\
&-\tfrac{1}{2\delta}\left(\norm{\bs \omega^k-\bar {\bs \omega}^k}^2+\norm{\bs \omega^k-\bs\omega^{k+1}}^2-\norm{\bs\omega^{k+1}-\bar {\bs \omega}^k}^2\right)
\end{aligned}
\end{equation*}
and we note that
\begin{equation*}\label{term_in_F}
\begin{aligned}
&\langle\Phi^{-1}\hat{\mc A}(\bs \omega^k,\xi^k),\bs\omega^*-\bs\omega^{k+1}\rangle\\
&=-\langle\Phi^{-1}\bar{\mc A}(\bs \omega^k),\bs\omega^k-\bs\omega^*\rangle+\langle \varepsilon^k,\bs\omega^*-\bs\omega^k\rangle+\\
&+\langle\Phi^{-1}\bar{\mc A}(\bs \omega^k),\bs \omega^k-\bs\omega^{k+1}\rangle+\langle \varepsilon^k,\bs \omega^k-\bs\omega^{k+1}\rangle.
\end{aligned}
\end{equation*}
Then, by reordering and substituting in (\ref{stepNk4}), we obtain
\begin{equation}\label{stepNk7}
\begin{aligned}
-&\norm{\bs\omega^{k+1}-\bar {\bs \omega}^k}^2-\norm{\bs\omega^{k+1}-\bs\omega^*}^2+\norm{\bs\omega^*-\bar {\bs \omega}^k}^2+\\
-&\tfrac{1}{\delta}\norm{\bs \omega^k-\bar {\bs \omega}^k}^2-\tfrac{1}{\delta}\norm{\bs \omega^k-\bs\omega^{k+1}}^2+\tfrac{1}{\delta}\norm{\bs\omega^{k+1}-\bar {\bs \omega}^k}^2+\\
-&2\langle\Phi^{-1}\bar{\mc A}(\bs \omega^k),\bs\omega^k-\bs\omega^*\rangle+2\langle \Phi^{-1}\varepsilon^k,\bs\omega^*-\bs\omega^k\rangle+\\
+&2\langle\Phi^{-1}(\bar{\mc A}(\bs \omega^k)-\bar{\mc A}(\bs\omega^{k-1})),\bs \omega^k-\bs\omega^{k+1}\rangle+\\
+&2\langle \Phi^{-1}(\varepsilon^k-\varepsilon^{k-1}),\bs \omega^k-\bs\omega^{k+1}\rangle\geq \alpha(H(\bs \omega^{k})-H(\bs\omega^*)).
\end{aligned}
\end{equation}
Since $\bar{\mc A}$ is monotone, it holds that
\begin{equation}\label{mono}
\begin{aligned}
&\langle\Phi^{-1}\bar{\mc A}(\bs \omega^k),\bs\omega^k-\bs\omega^*\rangle+\alpha (H(\bs \omega^{k})-H(\bs\omega^*))\geq\\
&\geq\langle\Phi^{-1}\bar{\mc A}(\bs \omega^*),\bs\omega^k-\bs\omega^*\rangle+\alpha (H(\bs \omega^{k})-H(\bs\omega^*))\geq0.
\end{aligned}
\end{equation}
Now we apply Lemma \ref{lemma_algo}(ii) and Lemma \ref{lemma_algo}(iii) to $\norm{\bs\omega^{k+1}-\bs\omega^*}$:
\begin{equation}\label{step_delta}
\begin{aligned}
\norm{\bs\omega^{k+1}-\bs\omega^*}^2
=&\tfrac{1}{1-\delta}\norm{\bar{\bs\omega}^{k+1}-\bs\omega^*}^2-\tfrac{\delta}{1-\delta}\norm{\bar{\bs\omega}^k-\bs\omega^*}^2+\\&+\delta\norm{\bs\omega^{k+1}-{\bs\omega}^k}^2.
\end{aligned}
\end{equation}
By substituting in (\ref{stepNk7}), 
grouping and reordering, we have
\begin{equation}\label{stepNk9}
\begin{aligned}
&\tfrac{1}{1-\delta}\norm{\bar{\bs\omega}^{k+1}-\bs\omega^*}^2
+\tfrac{1}{\delta}\norm{\bs \omega^k-\bs\omega^{k+1}}^2\leq\\
&\leq \left(\tfrac{\delta}{1-\delta}+1\right)\norm{\bs\omega^*-\bar {\bs \omega}^k}^2-\tfrac{1}{\delta}\norm{\bs \omega^k-\bar {\bs \omega}^k}^2+\\
&+2\langle\Phi^{-1}(\bar{\mc A}(\bs \omega^k)-\bar{\mc A}(\bs\omega^{k-1})),\bs \omega^k-\bs\omega^{k+1}\rangle\\
&+2\langle \Phi^{-1}\varepsilon^k,\bs\omega^*-\bs\omega^k\rangle+2\langle \Phi^{-1}(\varepsilon^k-\varepsilon^{k-1}),\bs \omega^k-\bs\omega^{k+1}\rangle
\end{aligned}
\end{equation}
where we used Assumption \ref{ass_delta}. 
Moreover, by using Lipschitz continuity of $\FF$ and Cauchy-Schwartz and Young's inequality, we obtain that
\begin{equation*}
\begin{aligned}
&\langle\Phi^{-1}(\bar{\mc A}(\bs \omega^k)-\bar{\mc A}(\bs\omega^{k-1})),\bs \omega^k-\bs\omega^{k+1}\rangle \leq \\
&\leq\frac{\ell_{\bar{\mc A}}\norm{\Phi^{-1}}}{2}\left(\normsq{\bs \omega^k-\bs\omega^{k-1}}+\normsq{\bs \omega^k-\bs\omega^{k+1}}\right).
\end{aligned}
\end{equation*}
Similarly, we can bound the term with the stochastic errors
\begin{equation*}
\begin{aligned}
&2\langle \Phi^{-1}(\varepsilon^k-\varepsilon^{k-1}),\bs \omega^k-\bs\omega^{k+1}\rangle\\
&\leq2\norm{\Phi^{-1}}\norm{\varepsilon^k-\varepsilon^{k-1}}\norm{\bs \omega^k-\bs\omega^{k+1}}\\
&\leq\norm{\Phi^{-1}}\normsq{\varepsilon^k-\varepsilon^{k-1}}+\norm{\Phi^{-1}}\normsq{\bs \omega^k-\bs\omega^{k+1}}.
\end{aligned}
\end{equation*}
Substituting in (\ref{stepNk9}), it yields
\begin{equation}\label{step_final}
\begin{aligned}
&\tfrac{1}{1-\delta}\norm{\bar{\bs\omega}^{k+1}-\bs\omega^*}^2+\tfrac{1}{\delta}\norm{\bs \omega^k-\bs\omega^{k+1}}^2\leq\\
&\leq\tfrac{1}{1-\delta}\norm{\bs\omega^*-\bar {\bs \omega}^k}^2-\tfrac{1}{\delta}\norm{\bs \omega^k-\bar {\bs \omega}^k}^2+\\
&+\ell_{\bar{\mc A}}\norm{\Phi^{-1}}\left(\normsq{\bs \omega^k-\bs\omega^{k-1}}+\normsq{\bs \omega^k-\bs\omega^{k+1}}\right)+\\
&+ \norm{\Phi^{-1}}\normsq{\varepsilon^k-\varepsilon^{k-1}}+\norm{\Phi^{-1}}\normsq{\bs \omega^k-\bs\omega^{k+1}}\\
&+2\langle \Phi^{-1}\varepsilon^k,\bs\omega^*-\bs\omega^k\rangle
\end{aligned}
\end{equation}
Now consider the residual function of $\bs\omega^k$:
\begin{equation*}
\begin{aligned}
\op{res}(\bs\omega^k)^2&=\normsq{\bs\omega_k-(\op{Id}+\Phi^{-1}\bar{\mc B})^{-1}(\bs\omega_k-\Phi^{-1}\bar{\mc A}(\bs\omega_k)}\\
&\leq 2\normsq{\bs\omega_k-\bs\omega^{k+1}}+2\normsq{\bar{\bs\omega}_k-\bs\omega_k+\Phi^{-1}\varepsilon_k}\\
&\leq2\normsq{\bs\omega_k-\bs\omega^{k+1}}+4\normsq{\bar{\bs\omega}_k-\bs\omega_k}+\normsq{\Phi^{-1}\varepsilon_k}
\end{aligned}
\end{equation*}
where we added and subtracted $\bs\omega^{k+1}=\op{proj}(\bar{\bs\omega}_k-\Phi^{-1}\hat{\mc A}(\bs\omega_k)$ in the first inequality and used the firmly nonexpansiveness of the projection and \eqref{norm_sum}.
It follows that
\begin{equation*}
\normsq{\bar{\bs\omega}_k-\bs\omega_k}\geq\tfrac{1}{4}\op{res}(\bs\omega^k)^2-\tfrac{1}{2}\normsq{\bs\omega_k-\bs\omega^{k+1}}-4\normsq{\Phi^{-1}\varepsilon_k}
\end{equation*}
Substituting in \eqref{step_final}
\begin{equation*}
\begin{aligned}
&\tfrac{1}{1-\delta}\norm{\bar{\bs\omega}^{k+1}-\bs\omega^*}^2+\tfrac{1}{\delta}\norm{\bs \omega^k-\bs\omega^{k+1}}^2\leq\tfrac{1}{1-\delta}\norm{\bs\omega^*-\bar {\bs \omega}^k}^2\\
&-\tfrac{1}{\delta}\left(\tfrac{1}{4}\op{res}(\bs\omega^k)^2-\tfrac{1}{2}\normsq{\bs\omega_k-\bs\omega^{k+1}}-\normsq{\Phi^{-1}\varepsilon_k}\right)+\\
&+\ell_{\bar{\mc A}}\norm{\Phi^{-1}}\left(\normsq{\bs \omega^k-\bs\omega^{k-1}}+\normsq{\bs \omega^k-\bs\omega^{k+1}}\right)+\\
&+ \norm{\Phi^{-1}}\normsq{\varepsilon^k-\varepsilon^{k-1}}+\norm{\Phi^{-1}}\normsq{\bs \omega^k-\bs\omega^{k+1}}+\\
&+2\langle \Phi^{-1}\varepsilon^k,\bs\omega^*-\bs\omega^k\rangle
\end{aligned}
\end{equation*}
By taking the expected value, grouping and using Proposition \ref{prop_variance} and Assumptions \ref{ass_error} and \ref{ass_step1}, we have
\begin{equation*}\label{stepNk10}
\begin{aligned}
&\EEk{\tfrac{1}{1-\delta}\norm{\bar{\bs\omega}^{k+1}-\bs\omega^*}^2}+\\
+&\EEk{\left(\tfrac{1}{2\delta}-\ell_{\bar{\mc A}}\norm{\Phi^{-1}}-\norm{\Phi^{-1}}\right)\norm{\bs \omega^k-\bs\omega^{k+1}}^2}\leq\\
\leq& \tfrac{1}{1-\delta}\norm{\bs\omega^*-\bar {\bs \omega}^k}^2+\ell_{\bar{\mc A}}\norm{\Phi^{-1}}\norm{\bs \omega^k-\bs\omega^{k-1}}^2+\\
&+\textstyle{\norm{\Phi^{-1}}C\left(\frac{2}{S_k}+\frac{2}{S_{k-1}}+\tfrac{1}{\delta}\frac{1}{S_k}\right)(\sigma(x^*)^2+\sigma_0^2\normsq{x-x^*})}\\
&-\tfrac{1}{\delta}\norm{\bs \omega^k-\bar {\bs \omega}^k}^2-\tfrac{1}{4\delta}\op{res}(\bs\omega^k)^2\\
\end{aligned}
\end{equation*}
To use Lemma \ref{lemma_RS}, let
$$\alpha_k= \tfrac{1}{1-\delta}\norm{\bs\omega^*-\bar {\bs \omega}^k}^2+\ell_{\bar{\mc A}}\norm{\Phi^{-1}}\norm{\bs \omega^k-\bs\omega^{k-1}}^2,$$
$$\theta_k=\tfrac{1}{\delta}\norm{\bs \omega^k-\bar {\bs \omega}^k}^2+\tfrac{1}{4\delta}\op{res}(\bs\omega^k)^2$$
$$\eta_k=\textstyle{\norm{\Phi^{-1}}C\left(\frac{2}{S_k}+\frac{2}{S_{k-1}}+\tfrac{1}{\delta}\frac{1}{S_k}\right)(\sigma(x^*)^2+\sigma_0^2\normsq{x-x^*}).}$$
By applying the Robbins Siegmund Lemma we conclude that $\alpha_k$ converges and that $\sum_k\theta_k$ is summable. This implies that the sequence $(\bar {\bs \omega}^k)_{k\in\NN}$ is bounded and that $\|\bs \omega^k-\bar {\bs \omega}^k\|\to 0$ (othewise $\sum \tfrac{1}{\delta}\|\bs \omega^k-\bar {\bs \omega}^k\|^2=\infty$). Therefore $(\bs \omega^k)_{k\in\NN}$ has at least one cluster point $\tilde{\bs\omega}$. Moreover, since $\sum\theta_k<\infty$, $\op{res}(\bs\omega^k)^2\to0$ and $\op{res}(\tilde{\bs\omega}^k)^2=0$.
\end{proof}

\begin{proof}[Proof of Corollary \ref{cor_coco}]
The proof is similar to Theorem \ref{theo_sgne} but we do not use monotonicity. Hence, the steps are the same, except for \eqref{mono}. Indeed, the terms in $H$ are not present since the projection satisfies \eqref{proj} and $\langle\bar{\mc A}(\bs \omega^k),\bs\omega^k-\bs\omega^*\rangle\!\geq\! 0$ by Assumption \ref{ass_weaker}. The conclusion follows as in Theorem \ref{theo_sgne}.
\end{proof}

\section{Proof of Theorem \ref{theo_coco}}\label{app_theo_coco}

\begin{proof}[Proof of Theorem \ref{theo_coco}]
The first part of the proof is the same as Theorem \ref{theo_sgne} since the resolvent is firmly nonexpansive but we do not use the residual nor monotonicity. Then, taking the expected value and grouping in \eqref{step_final}, we have
\begin{equation*}\label{stepNk10}
\begin{aligned}
&\EEk{\tfrac{1}{1-\delta}\norm{\bar{\bs\omega}^{k+1}-\bs\omega^*}^2}+\\
&+\EEk{\left(\tfrac{1}{\delta}-\ell_{\bar{\mc C}}\norm{\Phi^{-1}}-\norm{\Phi^{-1}}\right)\norm{\bs \omega^k-\bs\omega^{k+1}}^2}\leq\\
\leq &\tfrac{1}{1-\delta}\norm{\bs\omega^*-\bar {\bs \omega}^k}^2+\ell_{\bar{\mc C}}\norm{\Phi^{-1}}\norm{\bs \omega^k-\bs\omega^{k-1}}^2+\\
&-2\langle\Phi^{-1}\bar{\mc C}(\bs \omega^k),\bs\omega^k-\bs\omega^*\rangle-\tfrac{1}{\delta}\norm{\bs \omega^k-\bar {\bs \omega}^k}^2+\\
&+\textstyle{\left(\frac{2\norm{\Phi^{-1}}C}{S_k}+\frac{2\norm{\Phi^{-1}}C}{S_{k-1}}\right)}(\sigma(x^*)^2+\sigma_0^2\normsq{x-x^*}).
\end{aligned}
\end{equation*}
where the inequality follows by Proposition \ref{prop_variance} and Assumptions \ref{ass_error} and \ref{ass_step2}.
To use Lemma \ref{lemma_RS}, let
$$\alpha_k= \tfrac{1}{1-\delta}\norm{\bs\omega^*-\bar {\bs \omega}^k}^2+\ell_{\bar{\mc C}}\norm{\Phi^{-1}}\norm{\bs \omega^k-\bs\omega^{k-1}}^2,$$
$$\theta_k=\tfrac{1}{\delta}\norm{\bs \omega^k-\bar {\bs \omega}^k}^2+2\langle\Phi^{-1}\bar{\mc C}(\bs \omega^k),\bs\omega^k-\bs\omega^*\rangle,$$
$$\eta_k=2\norm{\Phi^{-1}} C\textstyle{\left(\frac{1}{S_k}+\frac{1}{S_{k-1}}\right)}(\sigma(x^*)^2+\sigma_0^2\normsq{x-x^*}).$$
Then, $\alpha_k$ converges and $\theta_k$ is summable. This implies that $\{\bar {\bs \omega}^k\}$ is bounded and that $\|\bs \omega^k-\bar {\bs \omega}^k\|\to 0$. Therefore $\{\bs \omega^k\}$ has at least one cluster point $\tilde{\bs\omega}$. Moreover, $\langle\Phi^{-1}\bar{\mc C}(\bs \omega^k),\bs \omega^k-\bs\omega^*\rangle\to 0$ and $\langle \bar{\mc C}(\tilde{\bs\omega}),\tilde{\bs\omega}-\bs\omega^*\rangle=0$. Since $\bar{\mc C}$ is cocoercive, it also satisfy the cut property, therefore $\tilde{\bs\omega}$ is a solution.
\end{proof}

\section{Proof of Theorem \ref{theo_sne}}\label{app_theo_sne}
\begin{proof}[Proof of Theorem \ref{theo_sne}]
We start by using the fact that the projection is firmly quasinonexpansive.
\begin{equation*}\label{sdet1}
\begin{aligned}
&\normsq{\bs x^{k+1}-\bs x^*}\leq\\
\leq&\normsq{\bs x^*-\bs {\bar x}^k+\lambda_k\hat F(\bs x^k,\xi^k)}-\normsq{\bar {\bs x}^k-\lambda_k\hat F(\bs x^k,\xi^k)-\bs x^{k+1}}\\
\leq&\normsq{\bs x-\bar{\bs x}^k}-\normsq{\bar{\bs x}^k-\bs x^{k+1}}+2\lambda_k\langle \hat F(\bs x^k,\xi^k),\bs x^*-\bar{\bs x}^k\rangle+\\
&+2\lambda_k\langle \hat F(\bs x^k,\xi^k),\bar{\bs x}^k-\bs x^{k+1}\rangle\\
=&\normsq{\bs x^*-\bar{\bs x}^k}-\normsq{\bar{\bs x}^k-\bs x^{k+1}}+2\lambda_k\langle \hat F(\bs x^k,\xi^k),\bar{\bs x}^k-\bs x^{k+1}\rangle\\
&+2\lambda_k\langle \hat F(\bs x^k,\xi^k),\bs x^*-\bs x^k\rangle+2\lambda_k\langle \hat F(\bs x^k,\xi^k),\bs x^k-\bar{\bs x}^k\rangle\\
\end{aligned}
\end{equation*}
In view of Lemmas \ref{lemma_algo}.2 and \ref{lemma_algo}.3 as in (\ref{step_delta}), we can rewrite the inequality as
\begin{equation}\label{sdet2}
\begin{aligned}
&\tfrac{1}{1-\delta}\normsq{\bar{\bs x}^{k+1}-\bs x^*}\leq\tfrac{1}{1-\delta}\normsq{\bar{\bs x}^k-\bs x^*}+\\
&+2\lambda_k\langle \hat F(\bs x^k,\xi^k),\bs x^*-\bs x^k\rangle+2\lambda_k\langle \hat F(\bs x^k,\xi^k),\bs x^k-\bar{\bs x}^k\rangle+\\
&+2\lambda_k\langle \hat F(\bs x^k,\xi^k),\bar{\bs x}^k-\bs x^{k+1}\rangle-(\delta+1)\normsq{\bs x^{k+1}-\bar{\bs x}^k}\\
\end{aligned}
\end{equation}
By applying the Young's inequality to the inner products we obtain
\begin{equation*}
\begin{aligned}
2\lambda_k\langle \hat F(\bs x^k,\xi^k),\bs x^k-\bar{\bs x}^k\rangle&\leq\lambda_k^2\normsq{\hat F(\bs x^k,\xi^k)}\!\!+\!\normsq{\bs x^k-\bar{\bs x}^k},\\
2\lambda_k\langle \hat F(\bs x^k,\xi^k),\bar{\bs x}^k-\bs x^{k+1}\rangle&\leq\lambda_k^2\normsq{\hat F(\bs x^k,\xi^k)}\!\!+\!\normsq{\bar{\bs x}^k-\bs x^{k+1}}
\end{aligned}
\end{equation*}
Then (\ref{sdet2}) becomes
\begin{equation}\label{sdet3}
\begin{aligned}
&\tfrac{1}{1-\delta}\normsq{\bar{\bs x}^{k+1}-\bs x^*}
\leq\tfrac{1}{1-\delta}\normsq{\bar{\bs x}^k-\bs x^*}+\\
&+2\lambda_k\langle \FF(\bs x^k),\bs x^*-\bs x^k\rangle+2\lambda_k\langle \epsilon^k,\bs x^*-\bs x^k\rangle+\\
&+4\lambda_k^2\normsq{\FF(\bs x^k)}+4\lambda_k^2\normsq{\epsilon^{k}}-(\delta+1)\normsq{\bs x^{k+1}-\bar{\bs x}^k}+\\
&+\normsq{\bs x^k-\bar{\bs x}^k}+\normsq{\bar{\bs x}^k-\bs x^{k+1}}\\
\end{aligned}
\end{equation}
By using Lemma \ref{lemma_algo}.1 and Assumption \ref{ass_error}, reordering and taking the expected value, we have
\begin{equation}\label{sdet4}
\begin{aligned}
&\EEk{\tfrac{1}{1-\delta}\normsq{\bar{\bs x}^{k+1}-\bs x^*}}+\EEk{\delta\normsq{\bs x^{k+1}-\bar{\bs x}^k}}\leq\\
&\leq\tfrac{1}{1-\delta}\normsq{\bar{\bs x}^k-\bs x^*}+\delta^2\normsq{\bs x^k-\bar {\bs x}^{k-1}}+\\
&+2\lambda_k\langle \FF(\bs x^k),\bs x^*-\bs x^k\rangle+4\lambda_k^2\normsq{\FF(\bs x^k)}+4\lambda_k^2\EEk{\normsq{\epsilon^k}}\\
\end{aligned}
\end{equation}
Thank to Lemma \ref{lemma_RS}, we conclude that $(\bar{\bs x}^k)_{k\in\NN}$ and $(\bs x^k)_{k\in\NN}$ are bounded sequence and that they have a cluster point, that is, $\bar{\bs x}^k\to\bar x$ and $\bs x^k\to x$. Since $\bar{\bs x}^{k} =(1-\delta) \bs x^k+\delta\bar{\bs x}^{k-1}$, taking the limit, we obtain that $\bar x=x$. Moreover, since $\langle \FF(\bs x^k),\bs x^*-\bs x^k\rangle\leq0$, again by Lemma \ref{lemma_RS}, we obtain that $\langle \FF(\bs x),\bs x-\bs x^*\rangle=0$ which, for the cut property, implies that $\bs x$ is a solution.
\end{proof}

\section{Proofs of Section \ref{sec_cor}}\label{app_sec_cor}

\begin{proof}[Proof of Corollary \ref{cor_weak}] We use the weak sharpness property in \eqref{sdet4} to obtain
\begin{equation*}
\begin{aligned}
&\EEk{\tfrac{1}{1-\delta}\normsq{\bar{\bs x}^{k+1}-\bs x^*}}+\EEk{\delta\normsq{\bs x^{k+1}-\bar{\bs x}^k}}\leq\\
&\leq\tfrac{1}{1-\delta}\normsq{\bar{\bs x}^k-\bs x^*}+\delta^2\normsq{x^k-\bar {\bs x}^{k-1}}+4\lambda_k^2\EEk{\normsq{\epsilon^k}}\\
&-2\lambda_kc\min_{\bs x^*\in \op{SOL}(\FF,\Omega)}\norm{\bs x-\bs x^*}+4\lambda_k^2\normsq{\FF(\bs x^k)}\\
\end{aligned}
\end{equation*}
Applying Lemma \ref{lemma_RS}, $(\bar{\bs x}^k)_{k\in\NN}$ and $(\bs x^k)_{k\in\NN}$ are bounded sequences and they have a cluster point $\bar{\bs x}$. Moreover, $\min_{\bs x^*\in \op{SOL}(\FF,\Omega)}\norm{\bs x-\bs x^*}\to0$ and $\norm{\bar{\bs x}-\bs x^*}=0$, that is, $\bar{\bs x}$ is a solution.
\end{proof}
\begin{proof}[Proof of the statement in Remark \ref{remark_acute}] We apply Lemma \ref{lemma_RS} to \eqref{sdet4}. Therefore, $(\bar{\bs x}^k)_{k\in\NN}$ and $(\bs x^k)_{k\in\NN}$ are bounded sequences and that they have a cluster point $\bar{\bs x}$. Moreover, $\langle \FF(\bs x^k),\bs x^*-\bs x^k\rangle\to0$ and $\langle \FF(\bs x),\bs x-\bs x^*\rangle=0$ but this contradicts the acute angle property. Therefore $\bar{\bs x}$ must be a solution.
\end{proof}

\balance
\bibliographystyle{IEEEtran}
\bibliography{Biblio}

%
%
%
%
%
%

\ifCLASSOPTIONcaptionsoff
  \newpage
\fi

%
%

\begin{IEEEbiography}[{\includegraphics[scale=.37]{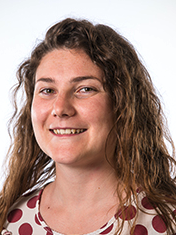}}]{Barbara Franci} is a PostDoc at the Delft Center for Systems and Control, Delft University of Technology, Delft, The Netherlands. 

She received the Bachelor's and Master's degree in Pure Mathematics from University of Siena, Siena, Italy, respectively in 2012 and 2014. Then, she received her PhD from Politecnico of Turin and University of Turin, Turin, Italy, in 2018. In September - December 2016 she visited the Department of Mechanical Engineering, University of California, Santa Barbara, USA.
She was awarded in 2017 with the Quality Award by the Academic Board of Politecnico di Torino.
\end{IEEEbiography}

\begin{IEEEbiography}[{\includegraphics[scale=.37]{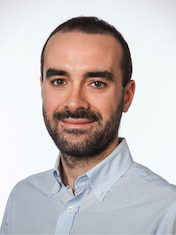}}]{Sergio Grammatico} is an Associate Professor at the Delft Center for Systems and Control, Delft University of Technology, The Netherlands. 

Born in Italy, in 1987, he received the Bachelor's degree (summa cum laude) in Computer Engineering, the Master's degree (summa cum laude) in Automatic Control Engineering, and the Ph.D. degree in Automatic Control, all from the University of Pisa, Italy, in February 2008, October 2009, and March 2013 respectively. He also received a Master's degree (summa cum laude) in Engineering Science from the Sant'Anna School of Advanced Studies, Pisa, Italy, in November 2011. 
In February-April 2010 and in November-December 2011, he visited the Department of Mathematics, University of Hawaii at Manoa, USA; in January-July 2012, he visited the Department of Electrical and Computer Engineering, University of California at Santa Barbara, USA. 

In 2013-2015, he was a post-doctoral Research Fellow in the Automatic Control Laboratory, ETH Zurich, Switzerland. In 2015-2018, he was an Assistant Professor, first in the Department of Electrical Engineering, Control Systems, TU Eindhoven, The Netherlands, then at the Delft Center for Systems and Control, TU Delft, The Netherlands. 
He was awarded a 2005 F. Severi B.Sc. Scholarship by the Italian High-Mathematics National Institute, and a 2008 M.Sc. Fellowship by the Sant’Anna School of Advanced Studies. He was awarded 2013 and 2014 “TAC Outstanding Reviewer” by the Editorial Board of the IEEE Trans. on Automatic Control, IEEE Control Systems Society. He was recipient of the Best Paper Award at the 2016 ISDG International Conference on Network Games, Control and Optimization.
\end{IEEEbiography}





\end{document}